\documentclass[a4paper,11pt,leqno]{article}
\usepackage{amsmath,amsfonts,amsthm,amssymb}
\usepackage[left=3cm, right=3cm, top=3cm, bottom=3.5cm]{geometry}
\usepackage{graphicx}
\numberwithin{equation}{section}

\theoremstyle{plain}
\newtheorem{theorem}{Theorem}[section]

\newtheorem{lemma}[theorem]{Lemma}
\newtheorem{proposition}[theorem]{Proposition}
\newtheorem{corollary}[theorem]{Corollary}

\theoremstyle{remark}
\newtheorem*{remark}{Remark}

\theoremstyle{definition}
\newtheorem*{definition}{Definition}

\newcommand{\R}{\mathbb{R}}
\newcommand{\SF}{\mathbb{S}}

\begin{document}

\title{Asymptotic behavior and rigidity results for symmetric solutions of the elliptic system $\Delta u=W_u(u)$ \\}

\author{Nicholas D. Alikakos\footnote{\rm The first author was partially supported through the project PDEGE – Partial Differential Equations
Motivated by Geometric Evolution, co-financed by the European Union – European Social
Fund (ESF) and national resources, in the framework of the program Aristeia of the ‘Operational
Program Education and Lifelong Learning’ of the National Strategic Reference Framework
(NSRF).}
\footnote{The research of N. Alikakos has been co-financed by the European Union – European Social
Fund (ESF) and Greek national funds through the �� Operational Program
 Education and Lifelong Learning’ of the National Strategic Reference
Framework (NSRF) - Research Funding Program: THALES}\ \ \ and Giorgio Fusco}

\date{}
\maketitle

\begin{abstract}
We study symmetric vector minimizers of the Allen-Cahn energy and establish various results concerning their structure and their asymptotic behavior.
\end{abstract}
 \section{Introduction}
 The problem of describing the structure of bounded solutions $u:\Omega\rightarrow\R^m$ of the equation
\begin{eqnarray}\label{system-0}
\left\{\begin{array}{l}
\Delta u=f(u),\quad x\in\Omega \\
u=u_0,\quad x\in\partial\Omega,
\end{array}\right.
\end{eqnarray}
where $f:\R^m\rightarrow\R^m$ is a smooth map and $\Omega\subset\R^n$ is a smooth domain that can be bounded or unbounded and may also enjoy symmetry properties, is a difficult and important problem which has attracted the interest of many authors in the last twenty five  years see \cite{gnn}, \cite{bcn}, \cite{bcn1} and \cite{eh} just to mention a few. Questions concerning monotonicity, symmetry and asymptotic behavior are the main objectives of these investigations. Most of the existing literature concerns the scalar case $m=1$ where a systematic use of the maximum principle and its consequences are the main tools at hand. For the vector case $m\geq 2$ we mention the works \cite{bgs} and \cite{gs}  where the control of the asymptotic behavior of solutions was basic for proving existence. In this paper we are interested in the case where $f(u)=W_u(u)$ is the gradient of a potential $W:\R^m\rightarrow\R$ and $u$ is a minimizer for the action functional $\int\frac{1}{2}\vert\nabla v\vert^2+W(v)$ in the sense of the following
\begin{definition}\label{definition-stable}
A map $u\in C^2(\Omega;\R^m)\cap L^\infty(\Omega;\R^m)$, $\Omega\subset\R^n$ an open set, is said to be a \underline{{\it minimizer}} or \underline{{\it minimal}} if for each bounded open lipshitz set $\Omega^\prime\subset\Omega$ it results
\begin{eqnarray}
J_{\Omega^\prime}(u)=\min_{v\in W_0^{1,2}(\Omega^\prime;\R^m)} J_{\Omega^\prime}(u+v),\quad\quad
J_{\Omega^\prime}(v)=\int_{\Omega^\prime}\frac{1}{2}\vert\nabla v\vert^2+W(v),
\end{eqnarray}
that is $u|_{\Omega^\prime}$ is an absolute minimizers in the set of $W^{1,2}(\Omega^\prime;\R^m)$ maps which coincide with $u$ on $\partial\Omega^\prime$.
\end{definition}
 Clearly if $u:\Omega\rightarrow\R^m$ is minimal then it is a solution of the Euler-Lagrange equation associated to the functional $J_{\Omega^\prime}$ which is the vector Allen-Cahn equation
\begin{equation}\label{system}
\Delta u=W_u(u),\quad x\in\Omega.
\end{equation}
We will work in the context of reflection symmetries.
Our main results are Theorem \ref{main} on the asymptotic behavior of symmetric minimizers and Theorem \ref{main-1} and Theorem \ref{triple} on the {\it rigidity} of symmetric minimizers. Rigidity meaning that, under suitable assumptions, a symmetric minimizer $u:\R^n\rightarrow\R^m$ must in effect depend on a number of variables $k<n$ strictly less than the dimension $n$ of the domain space. These theorems, in the symmetric setting, are  vector counterparts of analogous results which are well known in the scalar case $m=1$ \cite{bar} \cite{far}. However in the vector case there is more structure as we explain after the statement of Theorem \ref{main-2}. In \cite{af3} we discuss a rigidity theorem where the assumption of symmetry is removed.

We let $G$ a reflection group acting both on the domain space $\Omega\subseteq\R^n$ and on the target space $\R^m$. We assume that  $W:\R^m\rightarrow\R$ a $C^3$ potential such that
\begin{description}
\item[${\bf H}_1$]$W$ is symmetric with respect to $G$: $W(g u)=W(u),\;\text{ for }\;g\in G,\; u\in\R^m$.
\end{description}
For Theorem \ref{main} and Theorem \ref{main-1} $G=S$ the group of order $2$ generated by the reflection $\R^d\ni z\mapsto\hat{z}\in\R^d$ in the plane $\{z_1=0\}$:
\[\hat{z}=(-z_1,z_2,\ldots,z_d),\;d=n,\,m.\]
In this case the symmetry of $W$ is expressed by $W(\hat{u})=W(u),\;u\in\R^m$. For Theorem \ref{triple} $G=T$ the group of order $6$ of the symmetries of the equilateral triangle. $T$ is generated by the reflection $\gamma$ in the plane $\{z_2=0\}$ and $\gamma_\pm$ in the plane $\{z_2=\pm\sqrt{3}z_1\}$.
We let $F\subset\R^d$, $d=n$ or $d=m$ a fundamental region for the action of $G$ on $\R^d$. If $G=S$ we take $F=\R_+^d=\{z:z_1>0\}$. If $G=T$ we take $F=\{z:0<z_2<\sqrt{3}z_1,\;z_1>0\}$.
\begin{description}
\item[${\bf H}_2$] There exists $a\in\overline{F}$ such that:
\begin{eqnarray}
0=W(a)\leq W(u),\; u\in\overline{F}.
\end{eqnarray}
Moreover $a$ is nondegenerate in the sense that the quadratic form $D^2W(a)(z,z)$ is positive definite.
\end{description}

In the symmetric setting we assume minimality in the class of symmetric variations:
\begin{definition}\label{definition-stable-s} Assume that $\Omega\subset\R^n$ and
 $u\in C^2(\Omega;\R^m)\cap L^\infty(\Omega;\R^m)$, are symmetric
 \begin{equation}\label{symmetric-equiv}
 \begin{split}
 &x\in\Omega\Rightarrow\;g x \in\Omega,\;\text{ for }\;g\in G,\\
 &u(g x )= g u(x),\;\text{ for }\;g\in G,\;x\in\Omega.
 \end{split}
 \end{equation} Then $u$ is said to be a symmetric  minimizer  if for each bounded open symmetric lipschitz set $\Omega^\prime\subset\Omega$ and for each symmetric $v\in W_0^{1,2}(\Omega^\prime;\R^m)$ it results
\begin{eqnarray}
J_{\Omega^\prime}(u)\leq J_{\Omega^\prime}(u+v).
\end{eqnarray}
\end{definition}
In the following by a minimizer we will always mean a symmetric minimizer in the sense of the definition above.

\begin{theorem}\label{theorem-1}
Assume $G=S$ and assume that $W$ satisfies ${\bf H}_1-{\bf H}_2$. Assume that $\Omega\subseteq\R^n$ is {\it convex-symmetric} in the sense that
\begin{eqnarray}
     x=(x_1,\dots,x_n)\in\Omega\Rightarrow(t x_1,\dots,x_n)\in\Omega, \text{ for } \vert t\vert\leq 1.
\end{eqnarray}
Let $\mathcal{Z}=\{z\in\R^m:z\neq a, W(z)=0\}$ and let
$u:\Omega\rightarrow\R^m$ a minimizer that satisfies
\begin{eqnarray}
\vert u(x)-z\vert>\delta,\;\text{ for }\;z\in\mathcal{Z},\; d(x,\partial\Omega^+)\geq d_0,\;x\in\Omega^+,
\end{eqnarray}
$\Omega^+=\{x\in\Omega:x_1>0\}$, and
\begin{equation}\label{assumed-bound}
\vert u\vert+\vert\nabla u\vert\leq M,\;\text{ for }\;x\in\Omega,
\end{equation}
for some $M>0$

 Then there exist $k_0, K_0>0$ such that
\begin{eqnarray}\label{exponential-0}
\vert u-a\vert\leq K_0e^{-k_0 d(x,\partial\Omega^+)},\;\text{ for }\;x\in\Omega^+.
\end{eqnarray}
\end{theorem}
\begin{proof}
A minimizer $u$ satisfies the assumptions of Theorem $1.2$ in \cite{fu} that implies the result.
\end{proof}
Examples of minimizers that satisfy the hypothesis of Theorem \ref{theorem-1} are provided (see \cite{af2}) by the entire equivariant solutions of (\ref{system}) constructed in \cite{af1}, \cite{a1}, \cite{f}.
The gradient bound in (\ref{assumed-bound}) is a consequence of the smoothness of $\Omega$ or, as in the case of the entire solutions referred to above, follows from the fact that $u$ is the restriction to a non smooth set of a smooth map.

We denote $C_S^{0,1}(\overline{\Omega},\R^m)$ the set of lipschitz symmetric maps $v:\overline{\Omega}\rightarrow\R^m$ that satisfy the bounds
\begin{equation}\label{bounds}
\begin{split}
&\|v\|_{C^{0,1}(\overline{\Omega},\R^m)}\leq M,\\
&\vert v-a\vert +\vert\nabla v\vert\leq K_0e^{-k_0 d(x,\partial\Omega^+)},\;x\in\Omega^+.
\end{split}
\end{equation}
We remark that from (\ref{exponential-0}) and elliptic regularity, after redefining $k_0$ and $K_0$ if necessary, we have
\begin{equation}\label{gradu-expo}
u\in C_S^{0,1}(\overline{\Omega},\R^m),
\end{equation}
for the minimizer in Theorem \ref{theorem-1}.

\begin{theorem}\label{main}
Assume $W$, $\Omega$ and $u:\Omega\rightarrow\R^m$ as in Theorem \ref{theorem-1}. Assume moreover that
\begin{description}
\item[${\bf H}_3$] The problem
\begin{eqnarray}
\left\{\begin{array}{l}
 u^{\prime\prime}=W_u(u),\quad s\in\R \\
 u(-s)=\hat{u}(s),\;s\in\R,\\
 \lim_{s\rightarrow+\infty}u(s)=a,
\end{array}\right.
\end{eqnarray}
has a unique solution $\bar{u}:\R\rightarrow\R^m.$
\item[${\bf H}_4$] the operator $T$ defined by
\begin{eqnarray}
D(T)=W_S^{2,2}(\R,\R^m),\quad\quad Tv=-v^{\prime\prime}+W_{uu}(\bar{u})v,
\end{eqnarray}
where $W_S^{2,2}(\R,\R^m)\subset W^{2,2}(\R,\R^m)$ is the subspace of symmetric maps, has a trivial kernel.
\end{description}
Then there exist $k, K>0$ such that
\begin{eqnarray}\label{exp-baru}
\vert u(x)-\bar{u}(x_1)\vert\leq Ke^{-kd(x,\partial\Omega)},\quad x\in\Omega.
\end{eqnarray}
\end{theorem}
\begin{theorem}\label{main-1}
 Assume that $\Omega=\R^n$ and that $W$ and $u:\R^n\rightarrow\R^m$ are as in Theorem \ref{main}. Then
 $u$ is unidimensional:
\begin{equation}
u(x)=\bar{u}(x_1),\;x\in\R^n.\hskip3cm
\end{equation}
\end{theorem}
\begin{theorem}\label{main-2}
Assume $\Omega=\{x\in\R^n:\;x_n>0\}$, $W$ and $u:\Omega\rightarrow\R^m$ as in Theorem \ref{main}. Then
\[u(x)=\bar{u}(x_1),\;\text{ on }\;\partial\Omega\;\Rightarrow\;u(x)=\bar{u}(x_1),\;\text{ on }\;\Omega.\]
\end{theorem}

From \cite{af1}, \cite{a1} and \cite{f}, we know that given a finite reflection group $G$, provided $W$ is invariant under $G$, there exists a $G$-equivariant  solutions $u:\R^n\rightarrow\R^m$ of the system (\ref{system}). It is natural to ask about the asymptotic behavior of these solutions. In particular, given a unit vector $\nu=(\nu_1,\dots,\nu_n)\in\R^n$ one may wonder about the existence of the limit
\begin{eqnarray}\label{limit}
\lim_{\lambda\rightarrow +\infty}u(x^\prime+\lambda\nu)=\tilde{u}(x^\prime),
\end{eqnarray}
where $x^\prime$ is the projection of $x=x^\prime +\lambda\nu$ on the hyperplane orthogonal to $\nu$. One can conjecture that this limit does indeed exist and that $\tilde{u}$ is a solution of the same system equivariant with respect to the subgroup $G_\nu\subset G$ that leave $\nu$ fixed, the stabilizer of $\nu$. In \cite{af1}, \cite{a1} and \cite{f} an exponential estimate analogous to (\ref{exponential-0}) in Theorem \ref{theorem-1} was established. This gives a positive answer to this conjecture for the case where $\nu$ is inside the set $D=\text{Int}\cup_{g\in G_a}g\overline{F}$. Here $F$ is a fundamental region for the action of $G$ on $\R^d$, $d=n,\,m$ and $G_a\subset G$ is the subgroup that leave $a$ fixed. Under the assumptions ${\bf H}_3$ and ${\bf H}_4$ Theorem \ref{main} goes one step forward and shows that the conjecture is true when $\nu$ belongs to the interior of one of the walls of the set $D$ above and $G_\nu$ is the subgroup of order two generated by the reflection with respect to that wall. In the proof of Theorem \ref{main} the estimate (\ref{exponential-0}) is basic. Once the exponential estimate in Theorem \ref{main} is established, we conjecture that, under assumptions analogous to ${\bf H}_3$ and ${\bf H}_4$, the approach developed in the proof of Theorem \ref{main} can be used to handle the case where $\nu$ belongs to the intersection of two walls of $D$. We also expect that, under the assumption that at each step $\tilde{u}$ is unique and hyperbolic, the process can be repeated to show the whole hierarchy of limits corresponding to all possible choice of $\nu$ and always $\tilde{u}$ is a solution of the system equivariant with respect to the subgroup $G_\nu$. This program is motivated by the analogy between equivariant connection maps and minimal cones \cite{a2}. Theorem \ref{triple} below is an example of such a splitting result \cite{gmt} in the diffused interface set-up.
Our next result concerns minimizers equivariant with respect to the symmetry group $T$ of the equilateral triangle. We can imagine that $T=G_\nu$ for some $\nu$ that belongs to the intersection of two walls of $D$.
The following assumptions ${\bf H}^\prime_3$ and ${\bf H}^\prime_4$, in the case at hand $G=T$,  correspond to the assumption ${\bf H}_3$ and ${\bf H}_4$ in Theorem \ref{main}
\begin{description}
\item[${\bf H}^\prime_3$] The problem
\begin{eqnarray}
\left\{\begin{array}{l}
 u^{\prime\prime}=W_u(u),\quad s\in\R \\
 u(-s)=\gamma u(s),\;s\in\R,\\
 \lim_{s\rightarrow+\infty}u(s)=\gamma_\pm a,
\end{array}\right.
\end{eqnarray}
has a unique solution $\bar{u}:\R\rightarrow\R^m.$
\item[${\bf H}^\prime_4$] the operator $T$ defined by
\begin{eqnarray}
D(T)=W_\gamma^{2,2}(\R,\R^m),\quad\quad Tv=-v^{\prime\prime}+W_{uu}(\bar{u})v,
\end{eqnarray}
where $W_\gamma^{2,2}(\R,\R^m)\subset W^{2,2}(\R,\R^m)$ is the subspace of the maps that satisfy $u(-s)=\gamma u(s)$, has a trivial kernel.
\end{description}
Then we have the assumptions concerning uniqueness and hyperbolicity of $\tilde{u}$
\begin{description}
\item[${\bf H}_5$] There is a unique $G$-equivariant solution $\tilde{u}:\R^2\rightarrow\R^m$ of (\ref{system})
    \begin{equation}\label{g-equivariance}
    \tilde{u}(g s)=g \tilde u(s),\;\text{ for }\;g\in T,\;s\in\R^2
    \end{equation}
    that satisfies the estimate
\begin{equation}\label{exp-est-two}
\vert\tilde{u}(s)-a\vert\leq K e^{-kd(s,\partial D)},\;\text{ for }\;s\in\R^2,
\end{equation}
where $D=\mathrm{Int}\overline{F}\cup \gamma\overline{F}$.
\item[${\bf H}_6$] the operator $\mathcal{T}$ defined by
\begin{eqnarray}
D(\mathcal{T})=W_G^{2,2}(\R^2,\R^m),\quad\quad \mathcal{T}v=-\Delta v+W_{uu}(\bar{u})v,
\end{eqnarray}
where $W_T^{2,2}(\R^2,\R^m)\subset W^{2,2}(\R^2,\R^m)$ is the subspace of $T$-equivariant maps, has a trivial kernel.
\end{description}
We are now in the position of stating
\begin{theorem}\label{triple}
Assume that $W$ satisfies ${\bf H}_1$ and ${\bf H}_2$ with $a=(1,0)$ and moreover that $0=W(a)<W(u)$ for $u\in\overline{F}$. Assume that ${\bf H}^\prime_3$, ${\bf H}^\prime_4$ and ${\bf H}_5$, ${\bf H}_6$ hold.
Let $u:\R^n\rightarrow\R^m$, $n\geq 3$ and $m\geq 2$ be a $T$-equivariant minimizer that satisfies(\ref{assumed-bound}) and, for some $\delta, d_0>0$  the condition
\begin{equation}\label{stay-away}
\vert u(x)-\gamma_\pm a\vert\geq\delta\;\text{ for }\;d(x,\partial D)>d_0,\;x\in D,
\end{equation}
where $D=\{x\in\R^n: \vert x_2\vert<\sqrt{3} x_1,\;x_1>0\}$.

Then $u$ is two-dimensional:
\begin{equation}\label{two}
u(x)=\tilde{u}(x_1,x_2),\;x\in\R^n.
\end{equation}
\end{theorem}
\begin{remark}
If instead of a minimizers defined on $\R^n$ we had considered a minimizer defined on a subset $\Omega\subset\R^n$, instead of (\ref{two}), the conclusion of Theorem \ref{triple} would be exponential convergence of $u$ to $\tilde{u}$ similar to (\ref{exp-baru}).
\end{remark}
Theorem \ref{triple} is an example of a De Giorgi type result for systems where monotonicity is replaced by minimality ( see \cite{aac},\cite{jm} and section 3 in \cite{s}). It is the PDE analog of the fact that a minimal cone $\mathcal{C}$ in $\R^n$ with the symmetry of the equilateral triangle is necessarily of the form $\mathcal{C}=\tilde{\mathcal{C}}\times\R^{n-2}$, with $\tilde{\mathcal{C}}$ is the triod in the plane.
For De Giorgi type results for systems, for general solutions , but under
 monotonicity hypotheses on the potential W, we refer to Fazly and
 Ghoussoub \cite{fg}.
The rest of the paper is devoted to the proofs. In Section \ref{sec-main}  we prove Theorem \ref{main} in Section \ref{basic} and Section \ref{replacement-lemmas} we prove a number of Lemmas that are basic for the proof of Theorem \ref{main} that we conclude in Sections \ref{proof-main} and \ref{sec-exp}. Theorems \ref{main-1} and \ref{main-2} and Theorem \ref{triple} are proved in Section \ref{main-final} and Section \ref{sec-triple}.
\section{The proof of Theorem \ref{main}}\label{sec-main}
The proof of Theorem \ref{main} that we present here, from an abstract point of view, has a lot in common with the proof of Theorem 1.2 in \cite{fu}. We will remark on this point later and spend a few words to motivate the various lemmas that compose the proof of Theorem \ref{main}. We begin with some notation and two basic lemmas.
\subsection{Basic lemmas}\label{basic}
In the following we use the notation $x=(s,\xi)$ with $x_1=s$ and $(x_2,\dots,x_n)=\xi$.
From (\ref{bounds}) it follows that, if $(l,\xi)\in\Omega^+$ satisfies $ d((l,\xi),\partial\Omega^+)\geq l$, then the map $s\rightarrow u(s,\xi), s\in[-l,l],$ that we still denote with $u$ satisfies
the bound
\begin{eqnarray}\label{v-vs-exp-bound}
\vert u-a\vert+\vert u_s\vert\leq K_0e^{-k_0 s}, \text{ for } s\in[0,l].
\end{eqnarray}
We denote by $E_l^\mathrm{xp}\subset C^1([-l,l]:\R^m)$ the set of symmetric maps $v:[-l,l]\rightarrow\R^m$ that satisfy
\begin{equation}\label{define-exp}
\vert v\vert+\vert v_s\vert\leq K e^{-k s}, \text{ for } s\in[0,l]
\end{equation}
for some $k, K>0$. We refer to $E_l^\mathrm{xp}$ as the exponential class.

\noindent We let $T_l$ the operator defined by
\begin{eqnarray}
D_l(T_l)=\{v\in W_S^{2,2}([-l,l],\R^m):v(\pm l)=0\},\quad\quad T_lv=-v^{\prime\prime}+W_{uu}(\bar{u})v.
\end{eqnarray}

\noindent For $l\in(0,+\infty]$ we let $\langle v,w\rangle_l=\int_{-l}^lvw$ denote the inner product in $L^2((-l,l),\R^m)$. We let $\|v\|_l=\langle v, v\rangle_l^{\frac{1}{2}}$ and
$\|v\|_{1,l}=\|v\|_{W^{1,2}([-l,l],\R^m)}$.

\noindent For the standard inner product in $\R^m$ we use the notation $(\cdot,\cdot)$.

\noindent It follows directly from (\ref{define-exp}) that $\|v\|_{1,l}\leq C=\frac{K}{\sqrt{k}}$. We set
\begin{eqnarray}
 \mathcal{B}_l^{1,2}:=\{v\in W_S^{1,2}([-l,l],\R^m):
v(\pm l)=0;\;  \|v\|_{1,l}\leq C\},
\end{eqnarray}
where $W_S^{1,2}([-l,l],\R^m)$ is the subspace of symmetric maps. Let $\SF$ be defined by
\begin{eqnarray}
\SF=\{\nu\in W_S^{1,2}([-l,l],\R^m): \|\nu\|_l=1\}
\end{eqnarray}
and set $q_\nu=\max\{q:q\nu\in\mathcal{B}_l^{1,2}\}$.
\begin{lemma}\label{strict-minimizer}
Assume $H_1$ and $H_2$ as in Theorem \ref{main} and let ${\bf e}_l:\mathcal{B}_l^{1,2}\rightarrow\R$ be defined by
\begin{eqnarray}
{\bf e}_l(v):=\frac{1}{2}(\langle \bar{u}_s+v_s, \bar{u}_s+v_s\rangle_l- \langle \bar{u}_s, \bar{u}_s\rangle_l)+\int_{-l}^l (W(\bar{u}+v)-W(\bar{u})).
\end{eqnarray}
Then there exist $l_0>0,\, q^\circ >0 \text{ and } c>0$ such that, for all $l\geq l_0$, we have
\begin{eqnarray}\label{iota-properties}
\left\{\begin{array}{l}
D_{qq}{\bf e}_l(q\nu)\geq c^2,\quad \text{ for } q\in[0,q^\circ]\cap[0,q_\nu],\;\nu\in\SF,\\\\
{\bf e}_l(q\nu)\geq{\bf e}_l(q^\circ\nu),\;\, \text{ for }  q^\circ\leq q\leq q_\nu,\;\nu\in\SF,\\\\
 {\bf e}_l(q\nu)\geq \tilde{{\bf e}}_l(p,q,\nu):={\bf e}_l(p\nu)+D_q{\bf e}_l(p\nu)(q-p)
 ,\\\quad\hskip4cm \text{ for } 0\leq p<q\leq q_\nu\leq q^\circ,\;\nu\in\SF,\\\\
D_p\tilde{{\bf e}}_l(p,q,\nu)\geq 0 ,\quad \text{ for } 0\leq p<q\leq q_\nu\leq q^\circ,\;\nu\in\SF.
\end{array}\right.
\end{eqnarray}
\end{lemma}
\begin{remark}
 ${\bf e}_l$ is a kind of an {\it effective} potential. Indeed, as we shall see, in the proof of Theorem \ref{main} the map $L^2((-l,l),\R^m)\ni q\mapsto{\bf e}_l(q\nu)$ plays a role similar to the one of the usual potential $\R\ni q\mapsto W(a+q\nu)$ in the proof of Theorem 1.2 in \cite{fu}.
\end{remark}
\begin{proof}
By differentiating twice ${\bf e}_l(q\nu)$ with respect to $q$ gives
\begin{eqnarray}
D_{qq}{\bf e}_l(q\nu)&=&\int_{-l}^l(\nu_s,\nu_s)+\int_{-l}^lW_{uu}(\bar{u}+q\nu)(\nu,\nu)\\\nonumber
&=&
D_{qq}{\bf e}_l(q\nu)|_{q=0}+\int_{-l}^l(W_{uu}(\bar{u}+q\nu)-W_{uu}(\bar{u}))(\nu,\nu).
\end{eqnarray}
From the interpolation inequality:
\begin{equation}
\begin{split}
\|v\|_{L^\infty}\leq &\sqrt{2}\|v\|_{1,l}^{\frac{1}{2}}\|v\|_l^{\frac{1}{2}},\\
\leq &\sqrt{2}\|v\|_{1,l},
\end{split}
\end{equation}
for $q\nu\in\mathcal{B}_l^{1,2}$
we get via the second inequality
\begin{equation}
\|q\nu\|_{L^\infty}\leq \sqrt{2}C,
\end{equation}
and via the first
\begin{equation}
\|\nu\|_{L^\infty}\leq \sqrt{2}C^{\frac{1}{2}}q^{-\frac{1}{2}}.
\end{equation}
Therefore we have
\begin{eqnarray}\label{w-uu-estimate}
\vert W_{u_iu_j}(\bar{u}(s)+ q\nu(s))-W_{u_iu_j}(\bar{u}(s))\vert\leq
\sqrt{2}C^{\frac{1}{2}}\overline{W}^{\prime\prime\prime}q^{\frac{1}{2}},
\end{eqnarray}
where $\overline{W}^{\prime\prime\prime}$ is defined by
\begin{eqnarray}
\overline{W}^{\prime\prime\prime}:=\max_{\left.\begin{array}{l}
1\leq i,j,k\leq m\\
s\in\R, \vert\tau\vert\leq 1
 \end{array}\right.}W_{u_iu_ju_k}(\bar{u}(s)+ \tau\sqrt{2}C).
\end{eqnarray}
From (\ref{w-uu-estimate}) we get
\begin{eqnarray}\label{int-wuu-wuu-estimate}
\vert\int_{-l}^l(W_{uu}(\bar{u}+q\nu)-W_{uu}(\bar{u}))(\nu,\nu)\vert\leq
C_1q^\frac{1}{2},
\end{eqnarray}
where $C_1>0$ is a constant independent of $l$.
We now observe that
\begin{eqnarray}\label{tl-equal-t}
D_{qq}{\bf e}_l(q\nu)|_{q=0}=\langle T_l\nu,\nu\rangle_l=
\langle T\tilde{\nu},\tilde{\nu}\rangle_\infty,
\end{eqnarray}
where $\tilde{\nu}$ is the trivial extension of $\nu$ to $\R$. $T$ is a self-adjoint operator which is positive by the minimality of $\bar{u}$. Therefore assumption ${\bf H}_5$ implies that the point spectrum of $T$ is bounded below by a positive number. From ${\bf H}_2$ the smallest eigenvalue $\mu$ of the matrix $W_{uu}(a)$ is positive and Persson's Theorem in \cite{ag} implies that also the remaining part of the spectrum of $T$, the essential spectrum, is bounded below by $\mu>0$. It follows that the spectrum of $T$ is bounded below by a positive constant $0<\tilde{\mu}\leq\mu$. From this (\ref{tl-equal-t}) and Theorem 13.31 in \cite{r} it follows
\begin{eqnarray}
D_{qq}{\bf e}_l(q\nu)|_{q=0}\geq\tilde{\mu},
\end{eqnarray}
which together with (\ref{int-wuu-wuu-estimate}) implies
\begin{eqnarray}
D_{qq}{\bf e}_l(q\nu)|\geq\tilde{\mu}\geq c^2:=\frac{\tilde{\mu}}{2},\;\;\text{ for }\,q\in[0,\bar{q}]\cap[0,q_\nu],
\end{eqnarray}
where $\bar{q}=\frac{1}{4}\frac{\tilde{\mu}^2}{C_1}$.
This concludes the proof of (\ref{iota-properties})$_1$.
We now consider the problem
\begin{eqnarray}\label{constrained-minimization}
\min_{\left.\begin{array}{l}
v\in\mathcal{B}_l^{1,2}\\
\|v\|_l\geq \bar{q}
\end{array}\right.} {\bf e}_l(v)
\end{eqnarray}
Since the constraint in problem (\ref{constrained-minimization}) is closed with respect to weak convergence in $W_0^{1,2}$, if $\bar{v}_l$ is a minimizer of problem (\ref{constrained-minimization}), we have $\bar{v}_l\neq 0$.
 This implies
\begin{eqnarray}
{\bf e}_l(\bar{v}_l)=\alpha_l>0.
\end{eqnarray}
Indeed the uniqueness assumption about the minimizer $\bar{u}$ implies that $v\equiv 0$ is the unique minimizer of ${\bf e}_l$. We have
\begin{eqnarray}\label{alpha-infinity}
\liminf_{l\rightarrow+\infty}\alpha_l=\alpha>0.
\end{eqnarray}
To prove this we assume that instead there is a sequence $l_k$ such that $\lim_{k\rightarrow+\infty}\alpha_{l_k}=0$. We can also assume that the sequence
$\tilde{\bar{v}}_{l_k}$ of the trivial extensions of $\bar{v}_{l_k}$ converges weakly in $W^{1,2}$ to a map $\bar{v}$ which by lower semicontinuity satisfies
\begin{eqnarray}
{\bf e}_\infty(\bar{v})=0.
\end{eqnarray}
This is in contradiction with the assumption that $v\equiv 0$ is the unique minimizer of ${\bf e}_\infty$ indeed the constraint in problem (\ref{constrained-minimization}) persists in the limit and implies $\bar{v}\neq 0$. This establishes (\ref{alpha-infinity}) and concludes the proof of (\ref{iota-properties})$_2$ with $q^\circ=\min\{\bar{q},\alpha\}$.

\noindent The last two inequalities in (\ref{iota-properties}) are straightforward consequences of (\ref{iota-properties})$_1$.
\end{proof}

\begin{lemma}\label{l-infinity-less-l-2}
Let $u$ as in Theorem \ref{theorem-1} and assume that
\begin{eqnarray}\label{lemma-assumption}
(l,\xi)\in\Omega^+,\; d((l,\xi),\partial\Omega^+\geq l,
\end{eqnarray}
then
there is a constant $C_2>0$ independent of $l>1$,  such that
\begin{eqnarray}\label{l-infinity-less-l-2-1}
\|u(\cdot,\xi)-\bar{u}\|_{L^\infty([-l,l],\R^m)}\leq C_2
\| u(\cdot,\xi)-\bar{u}\|_l^\frac{2}{3}.
\end{eqnarray}
\end{lemma}
\begin{proof}
From (\ref{lemma-assumption}) $u(\cdot,\xi)$ satisfies (\ref{v-vs-exp-bound}). Since also $\bar{u}$ satisfies (\ref{v-vs-exp-bound}). There is $\bar{s}\in[0,l]$ such
that $\vert u(s,\xi)-\bar{u}(s)\vert\leq m=:\vert u(\bar{s},\xi)-\bar{u}(\bar{s})\vert$. From this and $\vert u(\cdot,\xi)_s-\bar{u}_s\vert\leq 2K_0$ it follows
\begin{eqnarray}
\vert u(s,\xi)-\bar{u}(s)\vert\geq m
(1-2K_0\vert s-\bar{s}\vert),\;\,\text{ for } s\in[-l,l]\cap[\bar{s}-\frac{m}{2K_0},\bar{s}+\frac{m}{2K_0}]
\end{eqnarray}
and a simple computation gives (\ref{l-infinity-less-l-2-1}).
\end{proof}

Before continuing with the proof, we explain the meaning of the lemmas that follow.
 Given $l, r>0$ and $\varsigma\in\R^{n-1}$ we let $\mathcal{C}_l^r(\varsigma)\subset\R^n$ the cylinder
\begin{eqnarray}
\mathcal{C}_l^r(\varsigma):=\{(s,\xi):-l<s<l;\,\vert \xi-\varsigma\vert<r\}.
\end{eqnarray}
Lemma \ref{lemma-1}, Lemma \ref{lemma-2} and Lemma \ref{lemma-w-q} describe successive deformations through which, fixed $\lambda>0$ and $\varrho>0$ and $\bar{q}\in(0,q^\circ)$, we transform the minimizer $u$ first into a map $v$ then into $w$ and finally into a map $w^{\bar{q}}$ that satisfies the conditions
\begin{equation}\label{conditions-w-q}
\begin{split}
& w^{\bar{q}}=u,\;\text{ on }\;\Omega\setminus\mathcal{C}_{l+\lambda}^{r+2\varrho}(\varsigma),\\
& w^{\bar{q}}(l+\frac{\lambda}{2},\xi)=\bar{u}(l+\frac{\lambda}{2}),\;\text{ for }\;\vert\xi-\varsigma\vert\leq r+\frac{\varrho}{2},\\
&\|w^{\bar{q}}(\cdot,\xi)-\bar{u}(\cdot)\|_{l+\frac{\lambda}{2}}\leq\bar{q},\;\text{ for }\;\vert\xi-\varsigma\vert\leq r+\frac{\varrho}{2}
\end{split}
\end{equation}
The deformations described in these lemmas are complemented by precise quantitative estimates on the amount of energy required for the deformation (see (iii) in Lemma \ref{lemma-1}, (iii) in Lemma \ref{lemma-2} and  (\ref{quantitative-w-wq}) in Lemma \ref{lemma-w-q}).
Lemma \ref{lemma-1} describes the deformation of $u$ into a map $v$ that  coincides with $\bar{u}$ on the lateral boundary of  $\mathcal{C}_{l+\frac{\lambda}{2}}^{r+\varrho}(\varsigma)$:
\begin{equation}\label{coincide-lateral}
\begin{split}
& v=u,\;\text{ outside }\;\mathcal{C}_{l+\lambda}^{r+2\varrho}(\varsigma)\setminus\overline{\mathcal{C}}_l^{r+2\varrho}(\varsigma)\\
&\|w (\cdot,\xi)-\bar{u}(\cdot)\|_{l+\frac{\lambda}{2}}\leq\bar{q},\;\text{ for }\;\vert\xi-\varsigma\vert= r+\frac{\varrho}{2}.
\end{split}
\end{equation}
Lemma \ref{lemma-2} describes the deformation of $v$ into a map $w$ that satisfies
\begin{equation}\label{coincide-l2}
\begin{split}
& w=v,\;\text{ outside }\;\mathcal{C}_{l+\frac{\lambda}{2}}^{r+\varrho}(\varsigma)\setminus\overline{\mathcal{C}}_{l+\frac{\lambda}{2}}^r(\varsigma)\\
&\|w (\cdot,\xi)-\bar{u}(\cdot)\|_{l+\frac{\lambda}{2}}\leq\bar{q},\;\text{ for }\;\vert\xi-\varsigma\vert= r+\frac{\varrho}{2}.
\end{split}
\end{equation}
Lemma \ref{quantitative-estimate0} and Corollary \ref{corollary} show that we can replace $w^{\bar{q}}$ with a map $\omega$ that coincides with $w^{\bar{q}}$ outside $\mathcal{C}_{l+\frac{\lambda}{2}}^{r+\frac{\varrho}{2}}(\varsigma)$ and has less energy than $w^{\bar{q}}$. Moreover Corollary \ref{corollary} yields a quantitative estimate for the energy difference.

In Sec.\ref{proof-main} we put together all these energy estimates and show (see Proposition \ref{l-2-bound}) that the assumption that
\[\|u (\cdot,\varsigma)-\bar{u}(\cdot)\|_l\geq q^\circ\]
if $r>0$ is sufficiently large, is incompatible with the minimality of $u$. Thus establishing that, if a sufficiently large cylinder $\mathcal{C}_{l+\frac{\lambda}{2}}^{r+\frac{\varrho}{2}}(\varsigma)$ is contained in $\Omega$, then we have the estimate
\[\|u (\cdot,\varsigma)-\bar{u}(\cdot)\|_l< q^\circ,\]
which is the main step in the proof of Theorem \ref{main}.

\subsection{Replacement Lemmas}\label{replacement-lemmas}

\begin{lemma}\label{lemma-1}
Let $\lambda \text{ and } \varrho>0$  be fixed. Assume that  $\mathcal{C}_{l+\lambda}^{r+2\varrho}(\varsigma)\subset\Omega$ satisfies
\begin{eqnarray}\label{distance-m-l}
d(\mathcal{C}_{l+\lambda}^{r+2\varrho}(\varsigma),\partial\Omega)\geq l+\lambda.
\end{eqnarray}
Then there exists a map $v\in C_S^{0,1}(\overline{\Omega},\R^m)$ such that
\begin{description}
\item[(i)] $v=u,\;\text{ on }\, \overline{\Omega}\setminus
(\mathcal{C}_{l+\lambda}^{r+2\varrho}(\varsigma)\setminus
\overline{\mathcal{C}}_{l}^{r+2\varrho}(\varsigma))$,
\item[(ii)] $v(l+\frac{\lambda}{2},\xi)=\bar{u}(l+\frac{\lambda}{2}),\;\text{ for }\, \vert\xi-\varsigma\vert\leq r+\varrho$.
\item[(iii)] $J_{\mathcal{C}_{l+\lambda}^{r+2\varrho}(\varsigma)}(v)-
J_{\mathcal{C}_{l+\lambda}^{r+2\varrho}(\varsigma)}(u)\leq C_0 r^{n-1}e^{-2k l}$,
\end{description}
where $C_0>0$ is a constant independent of $l$ and $r$.
\end{lemma}
\begin{proof}
For $(s,\xi)\in \overline{\mathcal{C}}_{l+\lambda}^{r+\varrho}(\varsigma)
\setminus\mathcal{C}_{l}^{r+\varrho}(\varsigma)$ we define $v$ by
\begin{eqnarray}\label{v-definition-1}
v(s,\xi)=
(1-\vert 1-2\frac{s-l}{\lambda}\vert)\bar{u}(s)+ \vert 1-2\frac{s-l}{\lambda}\vert u(s,\xi),\hskip2cm\\\nonumber\hskip4cm s\in[l,l+\lambda],\, \vert \xi-\varsigma\vert\leq r+\varrho.
\end{eqnarray}
It remains to define $v(s,\xi) \text{ for } (s,\xi)\in(l,l+\lambda)\times\{\xi:r+\varrho<\vert \xi-\varsigma\vert<r+2\varrho\}$.

Set
\begin{eqnarray}
B u(s,\xi)=\vert\frac{s-l-\lambda}{\lambda}\vert u(l,\xi)+\frac{s-l}{\lambda}u(l+\lambda,\xi),\\\nonumber
\tilde{u}(s,\xi)=u(s,\xi)-B u(s,\xi).\hskip2.5cm
\end{eqnarray}
Note that by (\ref{v-definition-1}) $\vert\xi-\varsigma\vert=r+\varrho$ implies $v(l,\xi)=u(l,\xi),\; v(l+\lambda,\xi)=u(l+\lambda,\xi)$ and therefore we have
\begin{eqnarray}
\vert\xi-\varsigma\vert=r+\varrho\Rightarrow B u(s,\xi)=B v(s,\xi),
\end{eqnarray}
where $v$ is defined in (\ref{v-definition-1}). Set
\begin{eqnarray}
\hat{v}(s,\xi)=
v(s,(r+\varrho)\frac{\xi-\varsigma}{\vert\xi-\varsigma\vert}+\varsigma)-
B u(s,(r+\varrho)\frac{\xi-\varsigma}{\vert\xi-\varsigma\vert}+\varsigma),
\end{eqnarray}
where again $v$ is defined in (\ref{v-definition-1}). With these notations we complete the definition of $v$ by setting
\begin{eqnarray}\label{v-definition-2}
v(s,\xi)=B u(s,\xi)
+\frac{\vert \xi-\varsigma\vert-r-\varrho}{\varrho}\tilde{u}(s,\xi)
+\frac{2\varrho+r-\vert \xi-\varsigma\vert}{\varrho}\hat{v}(s,\xi),\\\nonumber
\text{ for } (s,\xi)\in(l,l+\lambda)\times\{\xi:r+\varrho<\vert \xi-\varsigma\vert<r+2\varrho\}.
\end{eqnarray}
Statement (i) and (ii) are obvious consequences of the definition of $v$. Direct inspection of (\ref{v-definition-1}) and (\ref{v-definition-2}) shows that $v$ is continuous. From (\ref{v-definition-1}) $v(s,\xi)$ is a linear combination of $\bar{u}(s)$ and $u(s,\xi)$ computed for $s\in[l,l+\lambda]$. A similar statement applies to $v(s,\xi)$ in (\ref{v-definition-2}) since
$B u(s,\xi),\,\hat{v}(s,\xi)$ and $\tilde{u}(s,\xi)$ are linear combinations of $u(s,\xi)$ and $v(s,\xi)$ in (\ref{v-definition-1}) computed for $s\in[l,l+\lambda]$. From this, assumption (\ref{distance-m-l}) and (\ref{v-vs-exp-bound}) we conclude
\begin{eqnarray}\label{energy-density-bound}
\vert v-a\vert+\vert\nabla v\vert\leq C_3 e^{-k_0 l} \;\,\text{ for }
(s,\xi)\in \mathcal{C}_{l+\lambda}^{r+2\varrho}(\varsigma)\setminus
\overline{\mathcal{C}}_{l}^{r+2\varrho}(\varsigma),
\end{eqnarray}
where $C_3>0$ is a constant independent of $l$ and $r$. From (\ref{energy-density-bound}) and the assumptions on the potential $W$ it follows
\begin{eqnarray}
\frac{1}{2}\nabla v\vert^2+W(v)\leq C_4e^{-2k_0 l},
\end{eqnarray}
which together with $\mathcal{H}^n(\mathcal{C}_{l+\lambda}^{r+2\varrho}(\varsigma)\setminus
\overline{\mathcal{C}}_{l}^{r+2\varrho}(\varsigma))\leq C_5 r^{n-1}$ concludes the proof.
\end{proof}

Given a number $0<\bar{q}< q^\circ$, let $A_{\bar{q}}$ be the set
\begin{eqnarray}
A_{\bar{q}}:=\{\xi: \|v(\cdot,\xi)-\bar{u}(\cdot)\|_{l+\frac{\lambda}{2}}>\bar{q},\,\vert \xi-\varsigma\vert< r+\varrho\},
\end{eqnarray}
where $v$ is the map constructed in Lemma \ref{lemma-1}.

\begin{lemma}\label{lemma-2}
Let $v$ as before and let $S:=A_{\bar{q}}\cap\{\xi: r<\vert \xi-\varsigma\vert< r+\varrho\}$. Then there is a constant $C_1>0$ independent from $l \text{ and } r$ and a map $w\in C_S^{0,1}(\overline{\Omega},\R^m)$ such that
\begin{description}
\item[(i)] $w=v \text{ on } \overline{\Omega}\setminus(\mathcal{C}_{l+\frac{\lambda}{2}}^{r+\varrho}(\varsigma)
    \setminus\overline{\mathcal{C}}_{l+\frac{\lambda}{2}}^{r}(\varsigma))$
\item[(ii)] $\| w-\bar{u}\|_{l+\frac{\lambda}{2}}\leq \bar{q}, \text{ for } \vert\xi-\varsigma\vert=r+\frac{\varrho}{2}.$
\item[(iii)] $J_{\mathcal{C}_{l+\frac{\lambda}{2}}^{r+\varrho}(\varsigma)
    \setminus\overline{\mathcal{C}}_{l+\frac{\lambda}{2}}^{r}(\varsigma)}(w)-
J_{\mathcal{C}_{l+\frac{\lambda}{2}}^{r+\varrho}(\varsigma)
    \setminus\overline{\mathcal{C}}_{l+\frac{\lambda}{2}}^{r}(\varsigma)}(v)    \leq C_1\mathcal{H}^{n-1}(S)$.
    \end{description}
\end{lemma}

\begin{proof}
Set
\begin{equation}\label{qv-and-nuv}
\begin{split}
& q^v(\xi)=\|v(\cdot,\xi)-\bar{u}(\cdot)\|_{l+\frac{\lambda}{2}},\\
& \nu^v(s,\xi)=\frac{v(s,\xi)-\bar{u}(s)}{q^v(\xi)},
\end{split}\text{ for }\;s\in(-l-\frac{\lambda}{2},l+\frac{\lambda}{2}),\;\xi\in S.
\end{equation}
and, for $s\in(-l-\frac{\lambda}{2},l+\frac{\lambda}{2}),\;\xi\in S$, define
\begin{equation}\label{w}
\begin{split}
& w(s,\xi)=\bar{u}(s)+q^w(\xi)\nu^v(s,\xi),\\
& q^w(\xi)=(1-\vert 1-2\frac{\vert\xi-\varsigma\vert-r}{\varrho}\vert)\bar{q}+
\vert 1-2\frac{\vert\xi-\varsigma\vert-r}{\varrho}\vert q^v(\xi).
\end{split}
\end{equation}
From this definition it follows that $w$ coincides with $v=\bar{u}+q^v\nu^v$ if $\vert\xi-\varsigma\vert=r$ or $\vert\xi-\varsigma\vert=r+\varrho$ or $q^v=\bar{q}$. This shows that $w$ coincides with $v$ on the boundary of the set $(-l-\frac{\lambda}{2},l+\frac{\lambda}{2})\times S$ and proves (i). From (\ref{w}) also follows that $q^w=\bar{q}$ for $\vert\xi-\varsigma\vert=r+\frac{\varrho}{2}$ for $\xi\in S$. This and the definition of $S$ imply (ii). To prove (iii)
we note that
\begin{eqnarray}\label{w-bar-u}
\vert w-\bar{u}\vert=\vert q^w\nu^v\vert\leq\vert q^v\nu^v\vert=\vert v-\bar{u}\vert, \text{ for } s\in(-l-\frac{\lambda}{2},l+\frac{\lambda}{2}),\;\xi\in S.
\end{eqnarray}
which implies
\begin{eqnarray}
\vert w-a\vert\leq Ke^{-k s},\;\text{ for }\;s\in(0,l+\frac{\lambda}{2}),\;\xi\in S.
\end{eqnarray}
Therefore we have
\begin{eqnarray}\label{potential-bound}
\int_{-l-\frac{\lambda}{2}}^{l+\frac{\lambda}{2}}(W(w)-W(v))
\leq\int_{-l-\frac{\lambda}{2}}^{l+\frac{\lambda}{2}}W(w)\leq C,
 \text{ for } \xi\in S.
\end{eqnarray}
We can write
\[w=\frac{q^w}{q^v}(v-\bar{u}),\;\text{ for }\;s\in(0,l+\frac{\lambda}{2}),\;\xi\in S\]
therefore we have, using also (\ref{energy-density-bound})
\begin{equation}\label{computation}
\begin{split}
& w_s=\frac{q^w}{q^v}(v_s-\bar{u}_s)\;\Rightarrow\;\vert w_s\vert\leq K e^{-k\vert s\vert},\\
& w_{\xi_j}=(\frac{q^w}{q^v})_{\xi_j}(v-\bar{u})+\frac{q^w}{q^v}v_{\xi_j}.
\end{split}
\end{equation}
From $q^v_{\xi_j}=\langle \nu^v,v_{\xi_j}\rangle_{l+\frac{\lambda}{2}}$ and (\ref{w}) it follows
\begin{equation}\label{computation1}
\begin{split}
& (\frac{q^w}{q^v})_{\xi_j}=\vert 1-2\frac{\vert\xi-\varsigma\vert-r}{\varrho}\vert_{\xi_j}(1-\frac{\bar{q}}{q^v})
-(1-\vert 1-2\frac{\vert\xi-\varsigma\vert-r}{\varrho}\vert)\frac{\bar{q}}{(q^v)^2}\langle \nu^v,v_{\xi_j}\rangle_{l+\frac{\lambda}{2}},\\
& \Rightarrow\;\vert (\frac{q^w}{q^v})_{\xi_j}\vert\leq\frac{2}{\varrho}+\frac{1}{q^v}\|v_{\xi_j}\|_{l+\frac{\lambda}{2}}.
\end{split}
\end{equation}
where we have also used $\frac{\bar{q}}{q^v}\leq 1$ for $\xi\in S$. From (\ref{computation1}) and (\ref{computation1}) it follows
\[\vert w_{\xi_j}\vert\leq(\frac{2}{\varrho}+\frac{\|v_{\xi_j}\|_{l+\frac{\lambda}{2}}}{\bar{q}})\vert v-\bar{u}\vert+
\vert v_{\xi_j}\vert\leq K e^{-k\vert,\;\text{ for }\; s\vert}s\in(-l-\frac{\lambda}{2},l+\frac{\lambda}{2}),\;\xi\in S,\]
where we have also used (\ref{energy-density-bound}).
From this  and (\ref{computation}) we conclude
\begin{eqnarray}
\int_{-l-\frac{\lambda}{2}}^{l+\frac{\lambda}{2}}(\vert\nabla w\vert^2-\vert\nabla v\vert^2)\leq
\int_{-l-\frac{\lambda}{2}}^{l+\frac{\lambda}{2}}\vert\nabla w\vert^2\leq C,
 \text{ for } \xi\in S.
\end{eqnarray}
This inequality together with (\ref{potential-bound}) conclude the proof.
\end{proof}
\begin{lemma}\label{lemma-w-q}
Let $w$ the map constructed in Lemma \ref{lemma-2}. Define $w^{\bar{q}}$ by setting
\begin{eqnarray}
w^{\bar{q}}=\left\{\begin{array}{l}
\bar{u}+\bar{q}\nu^v, \text{ for } (s,\xi)\in\mathcal{C}_{l+\frac{\lambda}{2}}^{r+\frac{\varrho}{2}}(\varsigma),\;\xi\in A_{\bar{q}},\\\\
w, \text{ for } (s,\xi)\in\mathcal{C}_{l+\frac{\lambda}{2}}^{r+\frac{\varrho}{2}}(\varsigma),\;\xi\not\in A_{\bar{q}},
 \text{ and  for } (s,\xi)\not\in\mathcal{C}_{l+\frac{\lambda}{2}}^{r+\frac{\varrho}{2}}(\varsigma).
\end{array}\right.
\end{eqnarray}
Then $w^{\bar{q}}\in C_S^{0,1}(\overline{\Omega},\R^m)$ and
\begin{eqnarray}\label{quantitative-w-wq}
J_{\mathcal{C}_{l+\frac{\lambda}{2}}^{r+\frac{\varrho}{2}}(\varsigma)}(w^{\bar{q}})
-J_{\mathcal{C}_{l+\frac{\lambda}{2}}^{r+\frac{\varrho}{2}}(\varsigma)}(w)\leq 0.
\end{eqnarray}
\end{lemma}
\begin{proof}
We have $w-\bar{u}=q^w\nu^w$ and $q^w>\bar{q}$ on $A_{\bar{q}}$. Therefore, recalling the definition of ${\bf e}_l$ and Lemma \ref{strict-minimizer} we have
\begin{eqnarray}\label{difference-w-wq}
J_{\mathcal{C}_{l+\frac{\lambda}{2}}^{r+\frac{\varrho}{2}}(\varsigma)}(w^{\bar{q}})
-J_{\mathcal{C}_{l+\frac{\lambda}{2}}^{r+\frac{\varrho}{2}}(\varsigma)}(w)&=& \int_{\tilde{A}_{\bar{q}}}
({\bf e}_{l+\frac{\lambda}{2}}(\bar{q}\nu^w)-{\bf e}_{l+\frac{\lambda}{2}}(q^w\nu^w))d\xi\\\nonumber
&&+
\frac{1}{2}
\sum_j\int_{\tilde{A}_{\bar{q}}}(\langle w^{\bar{q}}_{\xi_j},w^{\bar{q}}_{\xi_j}\rangle_{l+\frac{\lambda}{2}}-
\langle w_{\xi_j},w_{\xi_j}\rangle_{l+\frac{\lambda}{2}})d\xi\\\nonumber
&\leq&
\frac{1}{2}
\sum_j\int_{\tilde{A}_{\bar{q}}}(\langle w^{\bar{q}}_{\xi_j},w^{\bar{q}}_{\xi_j}\rangle_{l+\frac{\lambda}{2}}-
\langle w_{\xi_j},w_{\xi_j}\rangle_{l+\frac{\lambda}{2}})d\xi,
\end{eqnarray}
To conclude the proof we note that for $\xi\in\tilde{A}_{\bar{q}}$
\begin{equation}\label{wj-expressions}
\begin{split}
& w_{\xi_j}^{\bar{q}}=\bar{q}\nu_{\xi_j}^v,\;\Rightarrow\;\langle w_{\xi_j}^{\bar{q}},w_{\xi_j}^{\bar{q}}\rangle_{l+\frac{\lambda}{2}}=\bar{q}^2\langle \nu_{\xi_j}^v,\nu_{\xi_j}^v\rangle_{l+\frac{\lambda}{2}},\\
& w_{\xi_j}=q_{\xi_j}^w\nu+q^w\nu_{\xi_j}^v,\;\Rightarrow\;\langle w_{\xi_j},w_{\xi_j}\rangle_{l+\frac{\lambda}{2}}= (q_{\xi_j}^w)^2+(q^w)^2\langle \nu_{\xi_j}^v,\nu_{\xi_j}^v\rangle_{l+\frac{\lambda}{2}}
\end{split}
\end{equation}
where we have also used that $\langle \nu^v,\nu_{\xi_j}^v\rangle_{l+\frac{\lambda}{2}}=0$. Form (\ref{wj-expressions}) it follows
\[\langle w^{\bar{q}}_{\xi_j},w^{\bar{q}}_{\xi_j}\rangle_{l+\frac{\lambda}{2}}-
\langle w_{\xi_j},w_{\xi_j}\rangle_{l+\frac{\lambda}{2}}=-(q_{\xi_j}^v)^2+(\bar{q}^2-(q^w)^2)\langle \nu_{\xi_j}^v,\nu_{\xi_j}^v\rangle_{l+\frac{\lambda}{2}}\leq 0,\]
for $\xi\in\tilde{A}_{\bar{q}}$. This and (\ref{difference-w-wq}) prove (\ref{quantitative-w-wq}).
 \end{proof}
Next we show that we can associate to $w^{\bar{q}}$ a map $\omega$ which coincides with $w^{\bar{q}}$ on $\Omega\setminus \mathcal{C}_{l+\frac{\lambda}{2}}^{r+\frac{\varrho}{2}}(\varsigma)$ and has less energy than $w^{\bar{q}}$. Moreover we derive a quantitative estimate of the energy difference. We follow closely the argument in \cite{fu}. First we observe that, if we define $q^\ast:=q^{w^{\bar{q}}}$, we can represent $J_{\mathcal{C}_{l+\frac{\lambda}{2}}^{r+\frac{\varrho}{2}}(\varsigma)}(w^{\bar{q}})$ in the {\it polar} form
\begin{eqnarray}
J_{\mathcal{C}_{l+\frac{\lambda}{2}}^{r+\frac{\varrho}{2}}(\varsigma)}(w^{\bar{q}})-
 J_{\mathcal{C}_{l+\frac{\lambda}{2}}^{r+\frac{\varrho}{2}}(\varsigma)}(\bar{u})
 \hskip7cm\\\nonumber
\hskip2cm=
\int_{B_{\varsigma,r +\frac{\varrho}{2}}\cap\{q^\ast>0\}}\frac{1}{2}(\vert\nabla q^\ast\vert^2+{q^\ast}^2\sum_j\langle \nu_{\xi_j}^w,\nu_{\xi_j}^w\rangle_{l+\frac{\lambda}{2}})+{\bf e}_{l+\frac{\lambda}{2}}(q^\ast\nu^w).
\end{eqnarray}
This follows from $\nu^w=\nu^v$ and from $\langle \nu^v,\nu_{\xi_j}^v\rangle_{l+\frac{\lambda}{2}}=0$ that implies
\[\sum_j\langle w_{\xi_j}^{\bar{q}},w_{\xi_j}^{\bar{q}}\rangle_{l+\frac{\lambda}{2}}=\vert\nabla q^\ast\vert^2+{q^\ast}^2\sum_j\langle \nu_{\xi_j}^w,\nu_{\xi_j}^w\rangle_{l+\frac{\lambda}{2}}\]
and from the definition of ${\bf e}_l$ in Lemma \ref{strict-minimizer}.
We remark that the definition of $q^\ast \text{ and } w^{\bar{q}}$ imply
\begin{eqnarray}
q^\ast&\leq& \bar{q}, \text{ on } B_{\varsigma,r +\frac{\varrho}{2}},\\\nonumber
q^\ast&=& \bar{q}, \text{ on } A_{\bar{q}}\cap B_{\varsigma,r +\frac{\varrho}{2}}. \end{eqnarray}

\begin{lemma}\label{quantitative-estimate0}
Let $\varphi:B_{\varsigma,r +\frac{\varrho}{2}}\rightarrow\R$ the solution of
\begin{eqnarray}\label{phi-comparison}
\left\{\begin{array}{l}
\Delta\varphi=c^2\varphi, \text{ in } B_{\varsigma,r +\frac{\varrho}{2}}\\
\varphi=\bar{q}, \text{ on } \partial B_{\varsigma,r +\frac{\varrho}{2}}.
\end{array}\right.
\end{eqnarray}
Then there is a map $\omega\in C_S^{0,1}(\overline{\Omega},\R^m)$ with the following properties
\begin{eqnarray}
\left\{\begin{array}{l}
\omega=w^{\bar{q}},\; \text{ on }\; \Omega\setminus \mathcal{C}_{l+\frac{\lambda}{2}}^{r+\frac{\varrho}{2}}(\varsigma),\\\\
\omega=q^\omega\nu^w+ \bar{u},\; \text{ on }\; \mathcal{C}_{l+\frac{\lambda}{2}}^{r+\frac{\varrho}{2}}(\varsigma),\\\\
q^\omega\leq\varphi\leq\bar{q},\;  \text{ on }\;  \mathcal{C}_{l+\frac{\lambda}{2}}^{r+\frac{\varrho}{2}}(\varsigma).
\end{array}\right.
\end{eqnarray}
Moreover
\begin{eqnarray}\label{quantitative-estimate}
J_{\mathcal{C}_{l+\frac{\lambda}{2}}^{r+\frac{\varrho}{2}}(\varsigma)}( w^{\bar{q}})-
J_{\mathcal{C}_{l+\frac{\lambda}{2}}^{r+\frac{\varrho}{2}}(\varsigma)}(\omega)\hskip7.5cm
\\\nonumber\hskip2cm\geq
\int_{B_{\varsigma,r +\frac{\varrho}{2}}\cap\{q^\ast>\varphi\}}
({\bf e}_{l+\frac{\lambda}{2}}(q^\ast\nu^w)-{\bf e}_{l+\frac{\lambda}{2}}(\varphi\nu^w)
-D_q{\bf e}_{l+\frac{\lambda}{2}}(\varphi\nu^w)(q^\ast-\varphi))d\xi.
\end{eqnarray}
\end{lemma}
\begin{proof}
 Let $b>0$,  $b\leq\min_{\xi\in B_{\varsigma,r +\frac{\varrho}{2}}}\varphi$ be fixed and let $A_b\subset B_{\varsigma,r +\frac{\varrho}{2}}$  the set $A_b:=\{\xi\in B_{\varsigma,r +\frac{\varrho}{2}}:q^\ast>b\}$.  $A_b$ is an open set since $w^{\bar{q}}=\bar{u}+q^\ast\nu^w$ is continuous by construction. Let
\begin{eqnarray}\label{reduced-action}
\mathcal{J}_{A_b}(p)=\int_{A_b}(\frac{1}{2}\vert\nabla p\vert^2+{\bf e}_{l+\frac{\lambda}{2}}(\vert p\vert\nu^w))d\xi,
\end{eqnarray}
Since $A_b$ is open and $q^\ast\in L^\infty(A_b,\R)$ there exists a minimizer $p^\ast\in q^\ast+W_0^{1,2}(A_b,\R)$ of the problem
\begin{eqnarray}
\mathcal{J}_{A_b}(p^\ast)=\min_{q^\ast+W_0^{1,2}(A_b,\R)}{\mathcal{J}_{A_b}}.
\end{eqnarray}
We also have
\begin{eqnarray}
0\,\leq\,p^\ast\,\leq\,\bar{q}.
\end{eqnarray}
This follows from (\ref{iota-properties}) that implies $\mathcal{J}_{A_b}(\frac{p^\ast+\vert p^\ast\vert}{2})\leq\mathcal{J}_{A_b}(p^\ast)$ and therefore $p^\ast\geq 0$. The other inequality is a consequence of  $\mathcal{J}_{A_b}(\min\{p^\ast,\bar{q}\})\leq\mathcal{J}_{A_b}(p^\ast)$ which follows from $\int_{A_b}\vert\nabla(\min\{p^\ast,\bar{q}\})\vert^2\leq
 \int_{A_b}\vert\nabla p^\ast\vert^2$ and from (\ref{iota-properties}).
Since the map $q\rightarrow {\bf e}_{l+\frac{\lambda}{2}}(\vert q\vert\nu^w))$ is a $C^1$ map, we can write the variational equation
\begin{eqnarray}\label{rho-variation0}
\int_{A_b}((\nabla p^\ast,\nabla\gamma)+D_q{\bf e}_{l+\frac{\lambda}{2}}(  p^\ast\nu^w)\gamma)d\xi=0,
\end{eqnarray}
for all $\gamma\in W_0^{1,2}(A_b,\R)\cap L^\infty(A_b)$. In particular, if we define $A_b^*:=\{x\in A_b: p^\ast>\varphi\}$, we have
\begin{eqnarray}\label{rho-variation}
\int_{A_b^*}((\nabla p^\ast,\nabla\gamma)+D_q{\bf e}_{l+\frac{\lambda}{2}}(  p^\ast\nu^w)\gamma)d\xi=0,
\end{eqnarray}
for all $\gamma\in W_0^{1,2}(A_b,\R)\cap L^\infty(A_b)$ that vanish on $A_b\setminus A_b^*$.
If we take $\gamma=(p^\ast-\varphi)^+$ in (\ref{rho-variation}) and use (\ref{iota-properties})$_2$ which implies $D_q{\bf e}_{l+\frac{\lambda}{2}}(  p^\ast\nu^w)\geq c^2 p^\ast$ we get
\begin{eqnarray}\label{rho-variation1}
\int_{A_b^*}((\nabla p^\ast,\nabla(p^\ast-\varphi))+c^2 p^\ast(p^\ast-\varphi))d\xi\leq 0,
\end{eqnarray}
This inequality and
\begin{eqnarray}\label{phi-variation1}
\int_{A_b^*}((\nabla\varphi,\nabla(p^\ast-\varphi))+c^2\varphi(p^\ast-\varphi))dx=0,
\end{eqnarray}
that follows from (\ref{phi-comparison}) imply
\begin{eqnarray}\label{rho-variation2}
\int_{A_b^*}(\vert\nabla(p^\ast-\varphi)\vert^2+c^2(p^\ast-\varphi)^2)d\xi\leq 0.
\end{eqnarray}
That is $\mathcal{H}^n(A_b^*)=0$
which together with $p^\ast\leq\varphi$ on $A_b\setminus A_b^*$ shows that
\begin{eqnarray}\label{rho-min-phi}
p^\ast\leq\varphi, \text{ for } \xi\in A_b.
\end{eqnarray}
Let $\omega$ be the map defined by setting
\begin{eqnarray}\label{v-definition}
\omega=\left\{\begin{array}{l}
w^{\bar{q}},\text{ for } (s,\xi)\in \Omega\setminus
 (-l-\frac{\lambda}{2}, l+\frac{\lambda}{2})\times A_b,\\\\
\bar{u}+q^\omega\nu^w=\bar{u}+\min\{p^\ast,q^\ast\}\nu^w, \text{ for } \xi\in A_b.
\end{array}\right.
\end{eqnarray}
Note that this definition, the definition of $A_b$ and (\ref{rho-min-phi}) imply
\begin{eqnarray}\label{q-omega-min-phi}
q^\omega\leq\varphi, \text{ for } \xi\in B_{\varsigma,r+\frac{\varrho}{2}}.
\end{eqnarray}
 From (\ref{v-definition}) we have
\begin{eqnarray}\label{diff-energy}
J_{\mathcal{C}_{l+\frac{\lambda}{2}}^{r+\frac{\varrho}{2}}(\varsigma)}( w^{\bar{q}})-
J_{\mathcal{C}_{l+\frac{\lambda}{2}}^{r+\frac{\varrho}{2}}(\varsigma)}(\omega)
\hskip8cm\\\nonumber\geq
 \int_{A_b\cap\{p^\ast<q^\ast\}}(\frac{1}{2}(\vert\nabla q^\ast\vert^2-\vert\nabla p^\ast\vert^2+((q^\ast)^2-(p^\ast)^2)\sum_{j=1}^{n} \langle   \nu_{\xi_j}^w, \nu_{\xi_j}^w \rangle_{l+\frac{\lambda}{2}})\\\nonumber+{\bf e}_{l+\frac{\lambda}{2}}(  q^\ast\nu^w)-{\bf e}_{l+\frac{\lambda}{2}}( p^\ast\nu^w))d\xi\\\nonumber \geq\int_{A_b\cap\{p^\ast<q^\ast\}}(\frac{1}{2}(\vert\nabla q^\ast\vert^2-\vert\nabla p^\ast\vert^2 +{\bf e}_{l+\frac{\lambda}{2}}(  q^\ast\nu^w)-{\bf e}_{l+\frac{\lambda}{2}}( p^\ast\nu^w))d\xi\\\nonumber \geq\int_{A_b\cap\{p^\ast<q^\ast\}}(\frac{1}{2}\vert\nabla q^\ast-\nabla p^\ast\vert^2\hskip5.5cm\\\nonumber  \hskip1.7cm+{\bf e}_{l+\frac{\lambda}{2}}(  q^\ast\nu^w)-{\bf e}_{l+\frac{\lambda}{2}}( p^\ast\nu^w)d\xi-D_q{\bf e}_{l+\frac{\lambda}{2}}( p^\ast\nu^w)(q^\ast-p^\ast))d\xi\geq 0.
\end{eqnarray}
where we have used
\begin{eqnarray}
\frac{1}{2}(\vert\nabla q^\ast\vert^2-\vert\nabla p^\ast\vert^2)&=&\frac{1}{2}\vert\nabla q^\ast-\nabla p^\ast\vert^2 +(\nabla p^\ast,\nabla( q^\ast-p^\ast)),\\\nonumber\\\nonumber \text{ and }\quad\quad\quad\quad\quad&&\\\nonumber\\\nonumber
\int_{A_b\cap\{p^\ast<q^\ast\}}(\nabla p^\ast,\nabla(q^\ast-p^\ast)) &=&-\int_{A_b\cap\{p^\ast<q^\ast\}}D_q{\bf e}_{l+\frac{\lambda}{2}}( p^\ast\nu^w)(q^\ast-p^\ast)d\xi,
\end{eqnarray}
which follows from (\ref{rho-variation0}) with $\gamma=(q^\ast-p^\ast)^+$.
From (\ref{iota-properties}$_3$) and  (\ref{rho-min-phi}) we have
\begin{eqnarray}
{\bf e}_{l+\frac{\lambda}{2}}(  q^\ast\nu^w)-\tilde{{\bf e}}_{l+\frac{\lambda}{2}}(p^\ast,q^\ast,\nu^w)\geq
{\bf e}_{l+\frac{\lambda}{2}}(  q^\ast\nu^w)-\tilde{{\bf e}}_{l+\frac{\lambda}{2}}(\varphi, q^\ast,\nu^w).
\end{eqnarray}
From this and (\ref{q-omega-min-phi}) which implies
\begin{eqnarray}
B_{\varsigma,r+\frac{\varrho}{2}}\cap\{\phi<q^\ast\}=
A_b\cap\{\phi<q^\ast\}\subset
A_b\cap\{p^\ast<q^\ast\},
\end{eqnarray}
we have
\begin{eqnarray}
 \int_{A_b\cap\{p^\ast<q^\ast\}} {\bf e}_{l+\frac{\lambda}{2}}(  q^\ast\nu^w)-{\bf e}_{l+\frac{\lambda}{2}}( p^\ast\nu^w)-D_q{\bf e}_{l+\frac{\lambda}{2}}( p^\ast\nu^w)(q^\ast-p^\ast)d\xi\\\nonumber
 \geq
  \int_{B_{\varsigma,r+\frac{\varrho}{2}}\cap\{\varphi<q^\ast\}} {\bf e}_{l+\frac{\lambda}{2}}(  q^\ast\nu^w)-{\bf e}_{l+\frac{\lambda}{2}}( p^\ast\nu^w)-D_q{\bf e}_{l+\frac{\lambda}{2}}( p^\ast\nu^w)(q^\ast-p^\ast)d\xi\\\nonumber
  \geq
    \int_{B_{\varsigma,r+\frac{\varrho}{2}}\cap\{\varphi<q^\ast\}} {\bf e}_{l+\frac{\lambda}{2}}(  q^\ast\nu^w)-{\bf e}_{l+\frac{\lambda}{2}}( \varphi\nu^w)-D_q{\bf e}_{l+\frac{\lambda}{2}}(\varphi\nu^w)(q^\ast-\varphi))d\xi.
\end{eqnarray}
The inequality (\ref{quantitative-estimate}) follows from this and (\ref{diff-energy}).
\end{proof}
\begin{corollary}\label{corollary}
Let $w^{\bar{q}}$ as before and let $\omega\in C_S^{0,1}(\overline{\Omega},\R^m)$ the map constructed in Lemma \ref{quantitative-estimate0}. Then there is a number $c_1>0$ independent from $l, r, \lambda$ and $\varrho$ such that
\begin{eqnarray}
J_{\mathcal{C}_{l+\frac{\lambda}{2}}^{r+\frac{\varrho}{2}}(\varsigma)}( w^{\bar{q}})-
J_{\mathcal{C}_{l+\frac{\lambda}{2}}^{r+\frac{\varrho}{2}}(\varsigma)}(\omega)
\geq
c_1\mathcal{H}^{n-1}(A_{\bar{q}}\cap B_{\varsigma,r}).
\end{eqnarray}
\end{corollary}
\begin{proof}
Set $R=r+\frac{\varrho}{2}$, then we have $\varphi(\xi)=\bar{q}\phi(\vert \xi-\varsigma\vert,R)$ with $\phi(\cdot,R):[0,R]\rightarrow\R$ a positive function which is strictly increasing in $(0,R]$. Moreover we have $\phi(R,R)=1$ and
\begin{eqnarray}\label{phi-l}
R_1<R_2,\;t\in(0, R_1)\;\Rightarrow\;\;\phi(R_1-t,R_1)>\phi(R_2-t,R_2).
\end{eqnarray}
Note that $\xi\in B_{\varsigma,r}$ implies $\varphi(\xi)\leq \bar{q}\phi(r,r+\frac{\varrho}{2})$.
Therefore for $\xi\in B_{\varsigma,r}\cap A_{\bar{q}}$ we have
 \begin{eqnarray}\label{diff-potential}\\\nonumber
 \hskip.5cm{\bf e}_{l+\frac{\lambda}{2}}(\bar{q}\nu^w)
 -{\bf e}_{l+\frac{\lambda}{2}}(\varphi\nu^w)
 -D_q{\bf e}_{l+\frac{\lambda}{2}}(\varphi\nu^w)(\bar{q}-\varphi)\hskip5cm\\\nonumber
 =\int_\varphi^{\bar{q}}(D_q{\bf e}_{l+\frac{\lambda}{2}}(s\nu^w)
 -D_q{\bf e}_{l+\frac{\lambda}{2}}(\varphi\nu^w))ds\hskip4.5cm\\\nonumber
 \hskip2cm\geq c^2\int_\varphi^{\bar{q}}(s-\varphi)ds=\frac{1}{2}c^2(\bar{q}-\varphi)^2\geq
 \frac{1}{2}c^2\bar{q}^2(1-\phi(r,r+\frac{\varrho}{2}))^2,
\end{eqnarray}
where we have also used (\ref{iota-properties})$_1$.
The corollary follows from this inequality, from (\ref{quantitative-estimate}) and from the fact that, by (\ref{phi-l}), the last expression in (\ref{diff-potential}) is increasing with $r$. Therefore, for $r\geq r_0$, for some $r_0>0$, we can assume
\begin{eqnarray}
c_1=\frac{1}{2}c^2\bar{q}^2(1
-\phi(r_0,r_0+\frac{\varrho}{2}))^2.
\end{eqnarray}
\end{proof}
\subsection{Conclusion of the proof of Theorem \ref{main}}\label{proof-main}
Let $u$ as in Theorem \ref{main} and $l_0,\,q^\circ$ as in Lemma \ref{strict-minimizer} and assume that $\varsigma$ is such that
\begin{eqnarray}
\| u(\cdot,\varsigma)-\bar{u}\|_l\geq q^\circ,
\end{eqnarray}
for some $l\geq l_0$.
Then $u\in C_S^{0,1}(\overline{\Omega},\R^m)$ implies that, there is  $r_0>0$ independent from $l\geq l_0$ such that,
\begin{eqnarray}\label{sigma-0}
\| u(\cdot,\xi)-\bar{u}\|_l\geq \bar{q},\, \text{ for } \vert \xi-\varsigma\vert\leq r_0.
\end{eqnarray}
Let $j_0\geq 0,$ be minimum value of $j$ that violated the inequality
\begin{eqnarray}\label{inequality}
c_1\frac{r_0^{n-1}}{2}(1+\frac{c_1}{C_1})^j\leq C_1((r_0+(j+1)\varrho)^{n-1}-(r_0+j\varrho)^{n-1}),
\end{eqnarray}
where $c_1 \text{ and } C_2$ are the constants in Corollary \ref{corollary} and Lemma \ref{lemma-2}.
Let $l^\circ\geq l_0$ be fixed so that
\begin{eqnarray}\label{l-0-sufficiently-large}
C_0(r_0+j_0\varrho)^{n-1}e^{-k l^\circ}\leq c_1\theta_{n-1}\frac{r_0^{n-1}}{2},
\end{eqnarray}
where $C_0$ is defined in Lemma \ref{lemma-1} and $\theta_n$ is the measure of the unit ball in $\R^n$,
\begin{proposition}\label{l-2-bound}
Let $\lambda, \varrho, \bar{q}\in(0,q^\circ)  \text{ and } l^\circ\geq l_0$ fixed as before and let $r^\circ=r_0+j_0\varrho$ where $j_0\geq 0$ is the minimum value of $j$ that violates (\ref{inequality}). Assume $l\geq l^\circ$ and assume that  $\mathcal{C}_{l+\lambda}^{r^\circ+2\varrho}(\varsigma)\subset\Omega$ satisfies
\begin{eqnarray}
d(\mathcal{C}_{l+\lambda}^{r^\circ+2\varrho}(\varsigma),\partial\Omega)\geq l+\lambda.
\end{eqnarray}
Then
\begin{eqnarray}
q^u(\varsigma)=\| u(\cdot,\varsigma)-\bar{u}\|_{l+\frac{\lambda}{2}}< q^\circ.
\end{eqnarray}

\end{proposition}
\begin{proof}
Suppose instead that
\begin{eqnarray}
\| u(\cdot,\varsigma)-\bar{u}\|_{l+\frac{\lambda}{2}} \geq q^\circ,
\end{eqnarray}
and set
\begin{eqnarray}\label{sigma-0-definition}
\sigma_0:=\theta_{n-1}\frac{r_0^{n-1}}{2}.
\end{eqnarray}
 Then $l^\circ\geq l_0$ and (\ref{sigma-0})) imply
\begin{eqnarray}
\mathcal{H}^{n-1}(A_{\bar{q}}\cap B_{\varsigma,r_0})\geq 2\sigma_0.
\end{eqnarray}
For each $0\leq j\leq j_0$ let $r_j:=r_0+j\varrho$ and let $v_j,\,w_j,\,w_j^{\bar{q}} \text{ and } \omega_j$ the maps $v,\,w,\,w^{\bar{q}} \text{ and } \omega$ defined in Lemma \ref{lemma-1},\,Lemma \ref{lemma-2},\, Lemma \ref{lemma-w-q} and Lemma \ref{quantitative-estimate0} with $l\geq l^\circ \text{ and } r=r_j$. Then from these Lemmas and Corollary \ref{corollary} we have
\begin{eqnarray}
\left.\begin{array}{l}
J(u)_{\mathcal{C}_{l+\lambda}^{r_j^\circ+2\varrho}(\varsigma)}
-J(v_j)_{\mathcal{C}_{l+\lambda}^{r_j^\circ+2\varrho}(\varsigma)}\geq-C_0r_j^{n-1}e^{-k l^\circ},\\\\
J(v_j)_{\mathcal{C}_{l+\lambda}^{r_j^\circ+2\varrho}(\varsigma)}
-J(w_j)_{\mathcal{C}_{l+\lambda}^{r_j^\circ+2\varrho}(\varsigma)}\geq-C_1\mathcal{H}^{n-1}(A_{\bar{q}}\cap (\overline{B}_{\varsigma,r_{j+1}}\setminus B_{\varsigma,r_j})),\\\\
J(w_j)_{\mathcal{C}_{l+\lambda}^{r_j^\circ+2\varrho}(\varsigma)}
-J(w_j^{\bar{q}})_{\mathcal{C}_{l+\lambda}^{r_j^\circ+2\varrho}(\varsigma)}\geq 0,\\\\
J(w_j^{\bar{q}})_{\mathcal{C}_{l+\lambda}^{r_j^\circ+2\varrho}(\varsigma)}
-J(\omega_j)_{\mathcal{C}_{l+\lambda}^{r_j^\circ+2\varrho}(\varsigma)}\geq c_1\mathcal{H}^{n-1}(A_{\bar{q}}\cap \overline{B}_{\varsigma,r_j}).
\end{array}\right.
\end{eqnarray}
From this and the minimality of $u$ it follows
\begin{eqnarray}\label{inequality-1}
\hskip1.5cm 0\geq -C_0r_j^{n-1}e^{-k l^\circ}-C_1\mathcal{H}^{n-1}(A_{\bar{q}}\cap (\overline{B}_{\varsigma,r_{j+1}}\setminus B_{\varsigma,r_j}))+ c_1\mathcal{H}^{n-1}(A_{\bar{q}}\cap \overline{B}_{\varsigma,r_j}).
\end{eqnarray}
Define
\begin{eqnarray}
\sigma_j:=\mathcal{H}^{n-1}(A_{\bar{q}}\cap B_{\varsigma,r_j})-\sigma_0, \text{ for } j\geq 1.
\end{eqnarray}
If $j_0=0$ the inequality (\ref{inequality-1}), using also (\ref{l-0-sufficiently-large}),  implies
\begin{eqnarray}\label{inequality-2}
0\geq -c_1\sigma_0-C_1\sigma_1+2C_1\sigma_0+2c_1\sigma_0\geq c_1\sigma_0-C_1(\sigma_1-\sigma_0).
\end{eqnarray}
If $j_0>0$ in a similar way we get
\begin{eqnarray}\label{inequality-3}
0\geq -c_1\sigma_0-C_1(\sigma_{j-1}-\sigma_j)+c_1(\sigma_j+\sigma_0)= c_1\sigma_j-C_1(\sigma_{j+1}-\sigma_j).
\end{eqnarray}
From (\ref{inequality-2}) and (\ref{inequality-3}) it follows
\begin{eqnarray}
\sigma_j\geq (1+\frac{c_1}{C_1})^j\sigma_0,
\end{eqnarray}
and therefore, using also (\ref{sigma-0-definition})
\begin{eqnarray}\label{inequality-4}
c_1(1+\frac{c_1}{C_1})^j\theta_{n-1}\frac{r_0^{n-1}}{2}\leq C_1(\sigma_{j+1}-\sigma{j})\leq C_1\theta_{n-1}(r_{j+1}^{n-1}-r_j^{n-1}).
\end{eqnarray}
This inequality is equivalent to (\ref{inequality}). It follows that, on the basis of the definition of $j_0$, putting $j=j_0$ in (\ref{inequality-4}) leads to a contradiction with the minimality of $u$.
\end{proof}
\subsection{The exponential estimate}\label{sec-exp}
\begin{lemma}\label{lemma-case-2}
Assume $r>r^\circ+2\varrho$ and $l>l^\circ+\lambda$ and assume that $\mathcal{C}_l^r(\varsigma_0)\subset\Omega$ satisfies
\begin{eqnarray}
d(\mathcal{C}_l^r(\varsigma_0),\partial\Omega)\geq l.
\end{eqnarray}
Then there are constants $K_1 \text{ and } k_1>0$ independent of $r>r^\circ+2\varrho$ and $l>l^\circ+\lambda$ such that
\begin{eqnarray}
\| u(\cdot,\varsigma_0)-\bar{u}\|_l^\frac{1}{2}
\leq K_1e^{-k_1r}.
\end{eqnarray}
\end{lemma}
\begin{proof}
From $r>r^\circ+2\varrho$ it follows that $\vert\varsigma-\varsigma_0\vert\leq r-(r^\circ+2\varrho)$ implies
\begin{eqnarray}
d(\mathcal{C}_l^{r^\circ+2\varrho}(\varsigma),\partial\Omega)\geq l.
\end{eqnarray}
Therefore we can invoke Proposition \ref{l-2-bound} to conclude that
\begin{eqnarray}
\| u(\cdot,\varsigma)-\bar{u}\|\leq \bar{q}, \text{ for } \vert\varsigma-\varsigma_0\vert\leq r-(r^\circ+2\varrho).
\end{eqnarray}
Let $\varphi:B_{\varsigma_0,r -(r^\circ+2\varrho)}\rightarrow\R$ the solution of
\begin{eqnarray}\label{phi-comparison-1}
\left\{\begin{array}{l}
\Delta\varphi=c^2\varphi, \text{ in } B_{\varsigma_0,r -(r^\circ+2\varrho)}\\\\
\varphi=\bar{q}, \text{ on } \partial B_{\varsigma_0,r -(r^\circ+2\varrho)}.
\end{array}\right.
\end{eqnarray}
Then we have
\begin{eqnarray}\label{q-omega-min-phi-1}
\| u(\cdot,\varsigma)-\bar{u}\|\leq \varphi(\varsigma), \text{ for } \varsigma\in B_{\varsigma_0,r -(r^\circ+2\varrho)}.
\end{eqnarray}
This follows by the same argument leading to (\ref{q-omega-min-phi}) in the proof of Lemma \ref{quantitative-estimate0}. Indeed, if (\ref{q-omega-min-phi-1}) does not hold, then by proceeding as in the proof of Lemma \ref{quantitative-estimate0} we can construct a competing map $\omega$ that satisfies (\ref{q-omega-min-phi-1}) and has less energy than $u$ contradicting its minimality property. In particular (\ref{q-omega-min-phi-1}) implies
\begin{eqnarray}\label{q-omega-min-phi-2}
\|u(\cdot,\varsigma_0)-\bar{u}\|\leq \varphi(\varsigma_0).
\end{eqnarray}
On the other hand it can be shown, see Lemma 2.4 in \cite{flp}, that there is a constant $h_0>0$ such that
\[\phi(0,r)\leq e^{-h_0 r};\;\text{ for }\;r\geq r_0\]
From this and (\ref{q-omega-min-phi-2}) we get
\begin{eqnarray}
\varphi(\varsigma_0)=\bar{q}\phi(0,r-(r^\circ+2\varrho))\leq
\bar{q}e^{h_0(r^\circ+2\varrho)}e^{-h_0r}=K_1e^{-k_1r}.
\end{eqnarray}
This concludes the proof with $K_1=\bar{q}e^{h_0(r^\circ+2\varrho)}$ and $k_1=h_0$.
\end{proof}

We are now in the position of proving the exponential estimate (i) in Theorem \ref{main}. We distinguish two cases:
\begin{description}
  \item[Case $1$] $x=(s,\xi)\in\Omega \text{ satisfies } s>\frac{1}{2}d(x,\partial\Omega)$. In this case, taking also into account that $\Omega$ satisfies $\bf{(i)}$, we have
      \begin{eqnarray}
      d(x,\partial\Omega^+)\geq \frac{1}{2}d(x,\partial\Omega).
      \end{eqnarray}
      From this and Theorem \ref{theorem-1} it follows
      \begin{eqnarray}\label{case-1-estimate}
      \vert u(s,\xi)-\bar{u}(s)\vert\leq \vert u(s,\xi)-a\vert+\vert \bar{u}(s)-a\vert\hskip4.5cm\\\nonumber\hskip2.5cm\leq K_0e^{-k_0d(x,\partial\Omega^+)}+\bar{K}e^{-\bar{k}s}
      \leq(K_0+\bar{K})e^{-\frac{1}{2}\min\{k_0,\bar{k}\}d(x,\partial\Omega)},
      \end{eqnarray}
      where we have also used
      \begin{eqnarray}
      \vert\bar{u}(s)-a\vert\leq\bar{K}e^{-\bar{k}s}.
      \end{eqnarray}
  \item[Case $2$] $x=(s,\xi)\in\Omega \text{ satisfies } 0\leq s\leq\frac{1}{2}d(x,\partial\Omega)$. In this case, elementary geometric considerations and the assumption $\bf{(i)}$ on $\Omega$ imply the existence of $\alpha\in(0,1)$ ($\alpha=\frac{1}{4}$ will do) such that
      \begin{eqnarray}\label{case-2-distance}
\mathcal{C}_{s+\alpha d(x)}^{\alpha d(x)}(\xi)&\subset&\Omega\hskip.8cm \text{ and }\\\nonumber
d(\mathcal{C}_{s+\alpha d(x)}^{\alpha d(x)}(\xi),\partial\Omega)&\geq& s+\alpha d(x),
\end{eqnarray}
where we have set $d(x):=d(x,\partial\Omega)$.
From (\ref{case-2-distance}) and Lemma \ref{lemma-case-2} it follows
\begin{eqnarray}
\| u(\cdot,\xi)-\bar{u}\|_l
\leq K_1e^{-k_1\alpha d(x)},\, \text{ for } d(x)>r^\circ+2\varrho.
\end{eqnarray}
This and Lemma \ref{l-infinity-less-l-2} imply, recalling $d(x)=d(x,\partial\Omega)$,
\begin{eqnarray}\label{case-2-exponential}
\vert u(s,\xi)-\bar{u}(s)\vert\leq K_1^\frac{2}{3}e^{-\frac{2}{3}k_1\alpha d(x,\partial\Omega)}.
\end{eqnarray}
\end{description}
The exponential estimate follows from  (\ref{case-2-exponential}) and  (\ref{case-2-exponential}).

\subsection{The proof of Theorems \ref{main-1} and \ref{main-2}}\label{main-final}
If $\Omega=\R^n$ the proof of Theorem \ref{main} simplifies since we can avoid the technicalities needed in the case that $\Omega$ is bounded in the $s=x_1$ direction and assume $l=+\infty$. The possibility of working with $l=+\infty$ is based on the following lemma
\begin{lemma}\label{l-infty}
Let $u:\R^n\rightarrow\R^m$ the symmetric minimizer in Theorem \ref{theorem-1}. Given a smooth open set $O\subset\R^{n-1}$ let $\R\times O$ the cylinder $\R\times O=\{(s,\xi): s\in\R,\;\xi\in\ O\}$. Then
\begin{equation}\label{min-infty}
J_{\R\times O}(u)=\min_{v\in u+W_{0 S}^{1,2}(\R\times O;\R^m)}J_{\R\times O}(v),
\end{equation}
where $W_{0 S}^{1,2}(\R\times O;\R^m)$ is the subset of $W_S^{1,2}(\R\times O;\R^m)$ of the maps that satisfy $v=0$ on $\partial\R\times O$.
\end{lemma}
\begin{proof}
Assume there are $\eta>0$ and $v\in W_{0 S}^{1,2}(\R\times O;\R^m)$ such that
\begin{equation}\label{cont-assum}
J_{\R\times O}(u)-J_{\R\times O}(v)\geq\eta.
\end{equation}
For each $l>0$ define $\tilde{v}\in W_{0 S}^{1,2}(\R\times O;\R^m)$ by
\[
\tilde{v}=\left\{\begin{array}{l} v,\quad\text{ for }\;s\in[0,l],\;\xi\in O,\\
(1+l-s)v+(s-l)u, \;s\in[l,l+1],\;\xi\in O,\\
u,\quad\text{ for } \;s\in[l,+\infty),\;\xi\in O.
\end{array}\right.\]
The minimality of $u$ implies
\begin{equation}\label{before-limit}
0\geq J_{[-l-1,l+1]\times O}(u)-J_{[-l-1,l+1]\times O}(\tilde{v})=J_{[-l-1,l+1]\times O}(u)-J_{[-l,l]\times O}(v)+
\mathrm{O}(e^{-k l}),
\end{equation}
where we have also used the fact that both $u$ and $v$ belong to $W_S^{1,2}(\R\times O;\R^m)$. Taking the limit for $l\rightarrow +\infty$ in (\ref{before-limit}) yields
\[0\geq J_{\R\times O}(u)-J_{\R\times O}(v)\]
in contradiction with (\ref{cont-assum}).
\end{proof}
Once we know that $u$ satisfies (\ref{min-infty}) the same arguments leading to Proposition \ref{l-2-bound} imply the existence of $r^\circ>0$ such that
\begin{equation}\label{stay-below-q0}
\R\times B_{r^\circ}(\xi)\subset\R^n\;\Rightarrow\;\|u(\cdot,\xi)-\bar{u}\|_{\infty}<q^\circ,
\end{equation}
where $B_{r^\circ}(\xi)\subset\R^{n-1}$ is the ball of center $\xi$ and radius $r^\circ$.
Since the condition $\R\times B_{r^\circ}(\xi)\subset\R^n$ is trivially satisfied for each $\xi\in\R^{n-1}$ we have
\[\|u(\cdot,\xi)-\bar{u}\|_{\infty}<q^\circ,\;\text{ for every }\;\xi\in\R^{n-1}.\]
To conclude the proof we observe that everything has been said concerning $q^\circ$ can be repeated verbatim for each $q\in(0,q^\circ)$. It follows that for each $q\in(0,q^\circ]$ there is a $r(q)>0$ such that (\ref{stay-below-q0}) holds with $q$ in place of $q^\circ$ and $r(q)$ in place of $r^\circ$. Therefore we have
\[\|u(\cdot,\xi)-\bar{u}\|_{\infty}<q,\;\text{ for every }\;\xi\in\R^{n-1}.\]
Since this holds for each $q\in(0,q^\circ]$ we conclude
\[u(\cdot,\xi)=\bar{u},\;\text{ for every }\;\xi\in\R^{n-1}\]
which complete the proof of Theorem \ref{main-1}.

To prove Theorem \ref{main-2} we note that, if $\Omega=\{x\in\R^n: x_n>0\}$, then arguing as in the proof of Theorem \ref{main-1} above, we get that, given $q>0$ there exists $l_q>0$ such that 
\[\xi_n> l_q,\quad\Rightarrow\quad\|u(\cdot,\xi)-\bar{u}\|_{L^\infty}<q.\]
From this, the boundary condition
\[\xi_n=0,\quad\Rightarrow\quad\|u(\cdot,\xi)-\bar{u}\|_{L^\infty}=0,\]
and the reasoning in the proof of Lemma \ref{lemma-w-q} it follows
\[\|u(\cdot,\xi)-\bar{u}\|_{L^\infty}<q,\;\text{ for each }\;\xi_n\geq 0,\,q>0.\]
The proof of Theorem \ref{main-2} is complete.
\section{The proof of Theorem \ref{triple}}\label{sec-triple}
From an abstract point of view the proof of Theorem \ref{triple} is essentially the same as the proof of Theorem \ref{main-1} after quantities like $q^u$ and $\nu^u$ are reinterpreted and properly redefined in the context of maps equivariant with respect to the group $G$ of the equilateral triangle. We divide the proof in steps pointing out the correspondence with the corresponding steps in the proof of Theorem \ref{main-1}. We write $x\in\R^n$ in the form $x=(s,\xi)$ with $s=(s_1,s_2)\in\R^2$ and $\xi=(x_2,\ldots,x_n)\in\R^{n-2}$.
\begin{description}
\item[Step 1]
\end{description}
From assumption (\ref{stay-away}) in Theorem \ref{triple} and equivariance it follows
\begin{equation}\label{stay-away-1}
\begin{split}
& \vert u(x)-a\vert\geq\delta,\;\vert u(x)-g_-a\vert>\delta,\;\text{ for }\;x\in g_+D,\;d(x,\partial g_+D)\geq d_0,\\
& \vert u(x)-a\vert\geq\delta,\;\vert u(x)-g_+a\vert>\delta,\;\text{ for }\;x\in g_-D,\;d(x,\partial g_-D)\geq d_0.
\end{split}
\end{equation}
From this and assumptions ${\bf H}^\prime_3$ and ${\bf H}^\prime_4$ it follows that we can apply Theorem \ref{main} with $\Omega=\R^n\setminus\overline{D}$ and $a_\pm=g_\pm a$ to conclude that there exist $k, K>0$ such that
\begin{equation}\label{exp-est-t}
\vert u(s_1,s_2,\xi)-\bar{u}(s_2)\vert\leq K e^{-k d(x,\partial(\R^n\setminus\overline{D}))},\;x\in\R^n\setminus\overline{D}.
\end{equation}
In exactly the same way we establish that
\begin{equation}\label{exp-est-2}
\vert \tilde{u}(s_1,s_2)-\bar{u}(s_2)\vert\leq K e^{-k d(s,\partial(\R^2\setminus\overline{D_2}))},\;s\in\R^2\setminus\overline{D_2},
\end{equation}
where $D_2\subset\R^2=\{s:\vert s_2\vert<\sqrt{3}s_1,\;s_1>0\}$.
From (\ref{exp-est-t}), (\ref{exp-est-2}) and equivariance it follows
\begin{equation}\label{exp-est-3}
\vert u(s,\xi)- \tilde{u}(s)\vert\leq K e^{-k\vert s\vert},\;\text{ for }\;s\in\R^2,\;\xi\in\R^{n-2}.
\end{equation}
\begin{description}
\item[Step 2]
\end{description}
Let $C_G^{0,1}(\R^n;\R^m)$ the set of  lipshizt maps $v:\R^n\rightarrow\R^m$ which are equivariant under $G$ and satisfy
\begin{equation}\label{exp-est-4}
\begin{split}
&\vert v(s,\xi)- \tilde{u}(s)\vert\leq K e^{-k\vert s\vert},\\
&\vert\nabla_s v(s,\xi)-\nabla_s\tilde{u}(s)\vert\leq K e^{-k\vert s\vert},\\
&\vert\nabla_\xi v(s,\xi)\vert\leq K e^{-k\vert s\vert},
\end{split}\;\text{ for }\;s\in\R^2,\;\xi\in\R^{n-2},
\end{equation}
We remark that from (\ref{exp-est-3}) we have $u\in C_G^{0,1}(\R^n;\R^m)$ for the minimizer $u$ in Theorem \ref{triple}. If $O\subset\R^{n-2}$ is an open bounded set with a lipshitz boundary we let $C_G^{0,1}(\R^2\times O;\R^m)$ the set of equivariant maps that satisfy (\ref{exp-est-4}) for $\xi\in O$. We denote $C_{0,G}^{0,1}(\R^2\times O;\R^m)$ the subset of $C_G^{0,1}(\R^2\times O;\R^m)$ of the maps the vanish on the boundary of $\R^2\times O$. The spaces $W_G^{1,2}(\R^2\times O;\R^m)$ and $W_{0,G}^{1,2}(\R^2\times O;\R^m)$ are defined in the obvious way. The exponential estimates in the definition of these function spaces and the same argument in the proof of Lemma \ref{l-infty} imply
\begin{lemma}\label{l-infty-t}
Let $u:\R^n\rightarrow\R^m$ the $G$-equivariant minimizer in Theorem \ref{triple}. Given an open bounded lipshitz set $O\subset\R^{n-2}$ we have
\begin{equation}\label{min-infty-t}
J_{\R^2\times O}(u)=\min_{v\in u+W_{0,G}^{1,2}(\R^2\times O;\R^m)}J_{\R^2\times O}(v),
\end{equation}
 \end{lemma}
\begin{description}
\item[Step 3]
\end{description}
 In analogy with the definition of  ${\bf e}(v)$ in Lemma \ref{strict-minimizer}, for $v\in W_G^{1,2}(\R^n;\R^m)$, we define the {\it effective} potential ${\bf E}(v)$ for the case at hand. We set
\begin{equation}\label{energy-t}
{\bf E}(v)=\frac{1}{2}(\langle\nabla_s\tilde{u}+\nabla_s v,\nabla_s\tilde{u}+\nabla_s v\rangle-\langle\nabla_s\tilde{u},\nabla_s\tilde{u}\rangle)+\int_{\R^2}(W(\tilde{u}+v)-W(\tilde{u}))ds,\;\xi\in\R^{n-2}.
\end{equation}
With this definition we can represent the energy $J_{\R^2\times O}(v)$ of a generic map $v\in W_G^{1,2}(\R^2\times O;\R^m)$ in the {\it polar} form
\begin{equation}\label{polar-form-t}
J_{\R^2\times O}(v)=\int_O\frac{1}{2}\big((\vert\nabla_\xi q^v\vert^2+(q^v)^2\sum_j\langle\nu_{\xi_j}^v,\nu_{\xi_j}^v\rangle)+
{\bf E}(q^v\nu^v)\big)d\xi,
\end{equation}
where $\langle,\rangle$ denotes the standard inner product in $L^2(\R^2;\R^m)$ and $q^v$ and $\nu^v$ are defined by
\begin{equation}\label{qn-nuv-t}
\begin{split}
& q^v(\xi)=\| v(\cdot,\xi)-\tilde{u}\|_{L^2(\R^2;\R^m)},\;\text{ for }\;\xi\in O\\
&\nu^v(s,\xi)=\frac{v(s,\xi)-\tilde{u}(s)}{q^v(\xi)},\;\text{ if }\;q^v(\xi)>0.
\end{split}
\end{equation}
From and assumptions ${\bf H}^\prime_5$ and ${\bf H}^\prime_5$, arguing exactly as in the proof of Lemma \ref{strict-minimizer} we prove
\begin{lemma}\label{strict-minimizer-t}
${\bf H}^\prime_5$ and ${\bf H}^\prime_5$.
Then there exist $q^\circ >0 \text{ and } c>0$ such that
\begin{eqnarray}\label{iota-properties-t}
\left\{\begin{array}{l}
D_{qq}{\bf E}(q\nu)\geq c^2,\quad \text{ for } q\in[0,q^\circ]\cap[0,q_\nu],\;\nu\in\SF,\\\\
{\bf E}(q\nu)\geq{\bf E}(q^\circ\nu),\;\, \text{ for }  q^\circ\leq q\leq q_\nu,\;\nu\in\SF,\\\\
 {\bf E}(q\nu)\geq \tilde{{\bf E}}(p,q,\nu):={\bf E}(p\nu)+D_q{\bf E}(p\nu)(q-p)
 ,\\\quad\hskip4cm \text{ for } 0\leq p<q\leq q_\nu\leq q^\circ,\;\nu\in\SF,\\\\
D_p\tilde{{\bf E}}(p,q,\nu)\geq 0 ,\quad \text{ for } 0\leq p<q\leq q_\nu\leq q^\circ,\;\nu\in\SF.
\end{array}\right.
\end{eqnarray}
\end{lemma}
\begin{description}
\item[Step 4]
\end{description}
Based on this lemma and on the polar representation of the energy (\ref{polar-form-t}) we can follow step by step  the arguments in Sec. 2 to establish the analogous of Proposition \ref{l-2-bound}. Actually the argument simplifies since by Lemma \ref{l-infty-t} we can work directly in $\R^2\times O$ rather then in bounded cylinders as in Sec. 2. For example the analogous of Lemma \ref{lemma-1} is not needed. In conclusion, by arguing as in Sec .2, we prove that, given $q\in(0,q^\circ]$, there is $r(q)>0$ such that
\begin{equation}\label{t-t}
\R^2\times B_{r(q)}(\xi)\subset\R^n\;\;\Rightarrow\;\;q^u(\xi)=\|u(\cdot,\xi)\tilde{u}\|_{L^2(\R^2;\R^m)}<q,
\end{equation}
where $B_{r(q)}(\xi)\subset\R^{n-2}$ is the ball of center $\xi$ and radius $r(q)$. Since the condition on the l.h.s. of (\ref{t-t}) is trivially satisfied for all $\xi\in\R^{n-2}$ and for all $q\in(0,q^\circ]$ we have
\[u(s,\xi)=\tilde{u}(s),\;\text{ for }\;s\in\R^2,\;\xi\in\R^{n-2}\]
which concludes the proof.
\bibliographystyle{plain}

\vskip.2cm
Department of Mathematics, University of Athens, Panepistemiopolis, 15784 Athens, Greece; e-mail: {\texttt{nalikako@math.uoa.gr}}
\vskip.2cm
\noindent Universit\`a degli Studi dell'Aquila, Via Vetoio, 67010 Coppito, L'Aquila, Italy; e-mail:{\texttt{fusco@univaq.it}}
\end{document}